\documentclass[final]{siamart1116}

\usepackage{mathtools}
\usepackage{amssymb}
\usepackage{stmaryrd}
\usepackage{color}
\usepackage{pdflscape}
\usepackage{epsfig}

\usepackage{epstopdf}
\usepackage{multirow}
\usepackage{hhline}
\usepackage{txfonts}
\usepackage{mathrsfs}
\usepackage{color}
\usepackage{algorithmic}

\usepackage{subcaption}
\usepackage{mathrsfs}

\newtheorem{thm}{Theorem}[section]

\newtheorem{lem}[thm]{Lemma}
\newtheorem{rem}[thm]{Remark}

\numberwithin{equation}{section}

\title{Operator-Based Uncertainty Quantification \\ of Stochastic Fractional PDEs \thanks{This work was supported by the AFOSR Young Investigator Program (YIP) award (FA9550-17-1-0150).}}

\author{
	Ehsan Kharazmi
	\footnote{D\lowercase{epartment of} C\lowercase{omputational} M\lowercase{athematics}, S\lowercase{cience}, \lowercase{and}, E\lowercase{ngineering} \& D\lowercase{epartment of} M\lowercase{echanical} E\lowercase{ngineering},	
		M\lowercase{ichigan} S\lowercase{tate} U\lowercase{niversity}, 428 S S\lowercase{haw} L\lowercase{ane}, E\lowercase{ast} L\lowercase{ansing}, MI 48824, USA}
	, Mohsen Zayernouri
	\footnote{D\lowercase{epartment of} C\lowercase{omputational} M\lowercase{athematics}, S\lowercase{cience}, \lowercase{and}, E\lowercase{ngineering} \&
		D\lowercase{epartment of} M\lowercase{echanical} E\lowercase{ngineering},	
		M\lowercase{ichigan} S\lowercase{tate} U\lowercase{niversity}, 428 S S\lowercase{haw} L\lowercase{ane}, E\lowercase{ast} L\lowercase{ansing}, MI 48824, USA,  C\lowercase{orresponding author; zayern@msu.edu}}
}

\begin{document}
\maketitle
\begin{abstract}
Fractional calculus provides a rigorous mathematical framework to describe anomalous stochastic processes by generalizing the notion of classical differential equations to their fractional-order counterparts. By introducing the fractional orders as uncertain variables, we develop an operator-based uncertainty quantification framework in the context of stochastic fractional partial differential equations (SFPDEs), subject to additive random noise. We characterize different sources of uncertainty and then, propagate their associated randomness to the system response by employing a probabilistic collocation method (PCM). We develop a fast, stable, and convergent Petrov-Galerkin spectral method in the physical domain in order to formulate the forward solver in simulating each realization of random variables in the sampling procedure.
\begin{keywords}
Stochastic fractional PDEs, forward uncertainty quantification, Monte Carlo simulation, probabilistic collocation method, Smolyak sparse grid, Petrov-Galerkin spectral method.
\end{keywords}
\end{abstract}
%
%

\pagestyle{myheadings}
\thispagestyle{plain}

%
\section{Introduction}
\label{Sec: Introduction}
%

Fractional models construct a tractable mathematical framework to describe and predict the behavior of multi-scales multi-physics complex phenomena. Particularly, fractional differential equations, as a well-structured generalization of their integer order counterparts, provide a rigorous mathematical tool to develop models, describing anomalous behavior in complex physical systems \cite{zhang2017review,jaishankar2014,jha2003evidence, Castillo2004Plasma,jaishankar2013,naghibolhosseini2015estimation,naghibolhosseini2017fractional,meer01, meral2010fractional}, where the anomalies manifest in heavy tailed distribution of corresponding underlying stochastic processes, moreover, exhibiting sharp peaks, intermittency, and asymmetry in the underlying distributions. Significant approximations as inherent part of assumptions upon which the model is built, lack of information about true values of parameters (incomplete data), and random nature of quantities being modeled pervade uncertainty in the corresponding mathematical formulations \cite{cullen1999probabilistic,roy2011comprehensive}. In this work, we develop an uncertainty quantification (UQ) framework in the context of stochastic fractional partial differential equations (SFPDEs), in which we characterize different sources of uncertainties and further propagate the associated randomness to the system response quantity of interest (QoI).

\vspace{0.1 in}
%
\noindent\textbf{Types and Sources of Uncertainty}.
%
The model uncertainties are in general being classified as aleatory, epistemic, and mixed, according to their fundamental essence. They can also be broadly characterized as occurring in model inputs, numerical approximations, and model form. Model inputs encompass all model parameters coming from geometry, constitutive laws, and fields equation, while also pertaining surrounding interactions, such as boundary conditions and random forcing sources (noise). Numerical approximations, which are an essence of differential equations since they generally do not lend themselves to analytical solutions, introduce uncertainty by imposing different sources of discretization error, iterative convergence error, and round off error. The fractional derivatives introduce derivative orders, namely fractional indices, as new set of model parameters in addition to model coefficients. These new parameters are strongly tied to the distribution of underlying stochastic process and their statistics are estimated from experimental observations in practice, see e.g.  \cite{benson2000application,baeumer2001subordinated}.

\vspace{0.1 in}
%
\noindent\textbf{Uncertainty Framework}.
%
Conventional approaches in parametric UQ of differential equations is based around Monte Carlo sampling (MCS) \cite{carlo1996concepts}, which performs ensemble of forward calculations to map the uncertain input space to the uncertain output space.
This method enjoys from being embarrassingly parallelizable and can be implement quite readily on high dimensional random spaces. However, the key issue is the slow rate of convergence $\sim 1/\sqrt{N}$ with $N$ number of realization, which consequently imposes exhaustively so many operations of forward solver, makes it not practical for expensive simulations. Other methods such as sequential MCS \cite{del2006sequential} and multilevel sequential MSC \cite{beskos2017multilevel} are also developed and recently used in \cite{jasra2016forward} to improve the parametric uncertainty assessment in elliptic nonlocal equations. An alternative to expensive MCS is to build surrogate models. An extensive comparison of two widely used ones, namely polynomial chaos and Gaussian process, are provided in a recent work \cite{owen2017comparison}. Polynomial chaos, in which the output of stochastic model is represented as a series expansion of input parameters, was initially applied in \cite{ghanem2003stochastic}, and later extended and used in \cite{xiu2002modeling,xiu2002wiener,knio2006uncertainty}. It is also generalized and used in constructing stochastic Galerkin methods \cite{babuska2004galerkin,babuvska2005solving,le2004uncertainty,le2004multi} for problems with higher-dimensional random spaces. Other non-sampling numerical methods, including but not limited to perturbation method \cite{schuss1980singular,babuvska2002solving,todor2005sparse,winter2002groundwater} and moment equation method \cite{liu1986probabilistic,liu1986random} are also developed, however their applications are restricted to stochastic systems with relatively low-dimensional random space. These so-called ``\textit{intrusive}" approaches typically do not treat the forward solver as a black-box, rather require some knowledge and reformulation of the governing equations and thus, may not be practical in many problems with complex codes.

A wide range of ``\textit{non-intrusive}" techniques mostly stretch over sampling, quadrature, and regression, see \cite{owen2017comparison} and references therein. More recently, high-order probabilistic collocation methods (PCM), employing the idea of interpolation/collocation in the random spaces, are developed in \cite{xiu2005high,babuvska2007stochastic,nobile2008sparse}. These methods limit the sample points to an efficient subset of random space, while adequately sampling the necessary range. The excellence in use of PCM is twofold; it has the benefit of easily sampling at nodal points that naturally leads to independent realizations of the problem as in MCS, and the advantage of fast convergence rate. The challenging post processing of solution statistics, which can essentially be described as a high-dimensional integration problem, can also be resolved by adopting sparse grid generators, such as Smolyak algorithm \cite{nobile2008sparse,Smolyak63}. The use of sparse grids, as opposed to full tensor product construction from one-dimensional quadrature rules, will effectively reduce the number of sampling, while preserving a fast convergence rate to high level of accuracy. 

\vspace{0.1 in}
%
\noindent\textbf{Forward Solver}.
%
A core task in computational forward UQ is to form an efficient numerical method, which for each realizations of random variables can accurately solves and simulates the deterministic counterpart of stochastic model in the physical domain. Such numerical method is usually called ``\textit{forward solver}" or ``\textit{simulator}". In the case of FPDEs, the excessive cost of numerical approximations due to the inherent nonlocal nature of fractional differential equations additionally become more challenging when generally most of uncertainty propagation techniques instruct operations of forward solver many times. This requires implementation of more efficient numerical schemes, which can manage increasing computational costs while maintaining sufficiently low error level in mitigating the corresponding uncertainties. In addition to numerous finite difference methods for solving FPDEs \cite{Gorenflo2002, Sun2006,  Lin2007, wang2010direct, wang2011fast, Cao2013, zeng2015numerical, zayernouri2016_JCP_Frac_AB_AM}, recent works have elaborated efficient spectral schemes, for discretizing FPDEs in physical domain, see e.g., \cite{Rawashdeh2006, Lin2007, Khader2011, Khader2012, Li2009, Li2010, chen2014generalized, wang2015high, bhrawy2015spectral}. More recently, Zayernouri et al. \cite{Zayernouri2013, zayernouri2015tempered} developed two new spectral theories on fractional and tempered fractional Sturm-Liouville problems, and introduced explicit corresponding eigenfunctions, namely \textit{Jacobi poly-fractonomials} of first and second kind. These eignefunctions are comprised of smooth and fractional parts, where the latter can be tunned to capture singularities of true solution. They are successfully employed in constructing discrete solution/test function spaces and developing a series of high-order and efficient Petrov-Galerkin spectral methods, see \cite{lischke2017petrov, suzuki2016fractional,samiee2016,samiee2017fast,samiee2017unified,samiee2018petrov,kharazmi2017petrov,kharazmi2017sem,kharazmi2018fractional}.

\vspace{0.1 in}

The main focus of this work is to develop an operator-based computational forward UQ framework in the context of stochastic fractional partial differential equation. Assuming that the mathematical model under consideration is well-posed and accounts in principle for all features of underlying phenomena, we identify three main sources of uncertainty, \textit{i}) parametric uncertainty, including fractional indices as new set of random parameters appeared in the operator, \textit{ii}) additive noises, which incorporates all intrinsic/extrinsic unknown forcing sources as lumped random inputs, and \textit{iii}) numerical approximations. Computational challenges arise when the admissible space of random inputs is infinite-dimensional, e.g. problems subject to additive noise \cite{rizzi2012uncertainty}, and thus, the framework involves uncertainty parametrization via a finite number of random space basis. Unlike the classical approach in modeling random inputs, which considers idealized uncorrelated processes (white noises), we model the random inputs as more/fully correlated random processes (colored noises), and parametrize them via Karhunen-Lo\`{e}ve (KL) expansion by assuming finite-dimensional noise assumption. This yields the problem in finite dimensional random space. We then, propagate the parametric uncertainties into the system response by applying PCM. We obtain the corresponding deterministic FPDE for each realization of random variables, using the Smolyak sparse grid generators for low to moderately high dimensions. In order to formulate the forward solver, we follow \cite{samiee2016} and develop a high-order Petrov-Galerkin (PG) spectral method to solve for each realization of SFPDE in the physical domain, employing Jacobi poly-fractonomials in addition to Legendre polynomials as temporal and spatial basis/test functions, respectively. The smart choice of coefficients in construction of spatial basis/test functions yields symmetric properties in the resulting mass/stiffness matrix, which is then exploited to formulate an efficient fast solver. We also show that for each realization of random variables, the deterministic problem is mathematically well-posed and the proposed forward solver is stable. By adopting sufficient number of basis in the physical domain, we eliminate the epistemic uncertainty associated with numerical approximation and isolate the impact of parametric uncertainty on system response QoI.


The organization of the paper is as follows. We recall some preliminaries on fractional calculus in section \ref{Sec: Fractional Calculus}. Then, we formulate the stochastic system in section \ref{Sec: Uncertainty framework}, and parametrize the random inputs. We also develop the stochastic sampling, namely PCM and MCS for our stochastic problem. We further construct the deterministic solver in section \ref{Sec: Forward Solver}, and provide the numerical results in section \ref{Sec: numerical results}. We end the paper with a conclusion and summary.


%
\section{Preliminaries on Fractional Calculus}
\label{Sec: Fractional Calculus}
%

Let $ \xi \in [-1,1]$. The left- and right-sided fractional derivative of order $\sigma$ are defined as (see e.g., \cite{Miller93, Podlubny99})
\begin{align}
\label{Eq: left RL derivative}
(\prescript{RL}{-1}{\mathcal{D}}_{\xi}^{\sigma}) u(\xi) = \frac{1}{\Gamma(n-\sigma)}  \frac{d^{n}}{d\xi^n} \int_{-1}^{\xi} \frac{u(s) ds}{(\xi - s)^{\sigma+1-n} },\quad \xi >-1 ,
\\
\label{Eq: right RL derivative}
(\prescript{RL}{\xi}{\mathcal{D}}_{1}^{\sigma}) u(\xi) = \frac{1}{\Gamma(n-\sigma)}  \frac{(-d)^{n}}{d\xi^n} \int_{\xi}^{1} \frac{u(s) ds}{(s - \xi)^{\sigma+1-n} },\quad \xi < 1 ,
\end{align}
respectively. 
An alternative approach in defining the fractional derivatives is the left- and right-sided Caputo derivatives of order $\sigma$, $n-1 < \sigma \leq n$, $n \in \mathbb{N}$, defined, as
\begin{align}
\label{Eq: left Caputo derivative}
&
(\prescript{C}{-1}{\mathcal{D}}_{\xi}^{\sigma} u) (\xi) = \frac{1}{\Gamma(n-\sigma)}  \int_{-1}^{\xi} \frac{u^{(n)}(s) ds}{(\xi - s)^{\sigma+1-n} },\quad \xi>-1,
\\
\label{Eq: right Caputo derivative}
&
(\prescript{C}{\xi}{\mathcal{D}}_{1}^{\sigma} u) (\xi) =  \frac{1}{\Gamma(n-\sigma)}  \int_{\xi}^{1} \frac{u^{(n)}(s) ds}{(s-\xi)^{\sigma+1-n} },\quad \xi<1.
\end{align}
By performing an affine mapping from the standard domain $[-1,1]$ to the interval $t \in [a,b]$, we obtain
\begin{eqnarray}
\label{Eq: RL in xL-xR}
\prescript{RL}{a}{\mathcal{D}}_{t}^{\sigma} u  &=&  (\frac{2}{b-a})^\sigma (\prescript{RL}{-1}{\mathcal{D}}_{\xi}^{\sigma} \, u )(\xi), 
\\ 
\label{Eq: Caputo in xL-xR}
\prescript{C}{a}{\mathcal{D}}_{t}^{\sigma} u  &=&  (\frac{2}{b-a})^\sigma (\prescript{C}{-1}{\mathcal{D}}_{\xi}^{\sigma} \, u) (\xi).
\end{eqnarray} 
Hence, we can perform the operations in the standard domain only once for any given $\sigma$ and efficiently utilize them on any arbitrary interval without resorting to repeating the calculations. Moreover, the corresponding relationship between the Riemann-Liouville and Caputo fractional derivatives in $[a,b]$ for any $\sigma \in (0,1)$ is given by 
\begin{equation}
\label{Eq:  Caputo vs. Riemann}
(\prescript{RL}{a}{\mathcal{D}}_{t}^{\sigma} \, u) (t)  =  \frac{ u(a)}{\Gamma(1-\sigma) (t-a)^{\sigma}}  +   (\prescript{C}{a}{\mathcal{D}}_{t}^{\sigma} \, u) (t).
\end{equation}
%

%
\section{Forward Uncertainty Framework}
\label{Sec: Uncertainty framework}
%

%
%
\subsection{Formulation of Stochastic FPDE}
\label{Sec: Problem Definition}
%
%
Let $\mathbb{D} = [0,T] \times [a_1,b_1] \times [a_2,b_2] \times \cdots \times [a_d,b_d]$ be the physical computational domain for some positive integer $d$ and stochastic function $u(t,\textbf{x};\omega): \mathbb{D} \times \Omega \rightarrow \mathbb{R}$, where $\omega \in \Omega$ denotes the random inputs of the system in a properly defined complete probability space $(\Omega, \mathcal{F},\mathbb{P})$. We consider the following SFPDE, subject to certain homogeneous Dirichlet initial/boundary conditions and stochastic process as additional force function, given as
\begin{align}
\label{Eq: Stochastic FPDE}
&
\mathcal{L}^{q(\omega)} \, u(t,\textbf{x};\omega) = F(t,\textbf{x};\omega)
\\
\label{Eq: IC}
& u \big\arrowvert_{t=0} = 0 ,  
\\
\label{Eq: BC}
& u \big\arrowvert_{x=a_j} = u \big\arrowvert_{x=b_j} = 0 , 
\end{align}
such that for $\mathbb{P}$-almost everywhere $\omega \in \Omega $ the equation holds. The stochastic fractional operator and force term, are given respectively as:
\begin{align}
\label{Eq: Stochastic Fractional Operator}
&
\mathcal{L}^{q(\omega)} \, = 
\prescript{}{0}{\mathcal{D}}_{t}^{\alpha(\omega)}
- \sum_{j=1}^{d} \, k_{j} \,\left[ \prescript{}{a_j}{\mathcal{D}}_{x_j}^{\beta_j(\omega)}
+ \prescript{}{x_j}{\mathcal{D}}_{b_j}^{\beta_j(\omega)} \right]
\\
& 
F(t,\textbf{x};\omega) = h(t,\textbf{x}) +  f(t;\omega) ,
\end{align}
where the fractional indices $\alpha(\omega) \in (0,1)$ and $\beta_j(\omega) \in (1,2), \,\, j=1,2,\cdots d$ are mutually independent random variables, $k_{j}$ are real positive constant coefficients, and the fractional derivatives are taken in the Riemann-Liouville sense. 
We assume that the driving terms $h$ and $f$ are properly posed, such that \eqref{Eq: Stochastic FPDE}-\eqref{Eq: BC} is well-posed $\mathbb{P}$-a.e. $\omega \in \Omega$, and also the solution in physical domain $\mathbb{D}$ is smooth enough such that we can construct a numerical scheme to solve each realization of SFPDE. As an extension to future works, the stochastic operator \eqref{Eq: Stochastic Fractional Operator} can be extended to $\alpha(\omega) \in (1,2)$ for the case of wave equations, and thus applied in formulating fractional models to study complex time-varying nonlinear fluid-solid interaction phenomena \cite{atanackovic2014,afzali2016vibrational,afzali2017analysis} and also the effect of damping in structural vibrations \cite{zamani2015asymmetric}.

%
%
\subsection{Representation of the Noise: Dimension Reduction}
\label{Sec: Noise}
%
%
We approximate the additional random forcing term by representing $f(t;\omega)$ into its finite dimensional version and thus, reduce the infinite-dimensional probability space to a finite-dimensional space. This is achieved via truncating Karhunen-Lo\`{e}ve (KL) expansion with the desired accuracy\cite{Loeve77}. 

Let $(\Omega, \mathcal{F},\mathbb{P})$ be a complete probability space, where $\Omega$ is the space of events, $\mathcal{F} \subset 2^{\Omega}$ denotes the $\sigma$-algebra of sets in $\Omega$, and $\mathbb{P}$ is the probability measure. The random field $f(t;\omega)$ has the ensemble mean $\mathbb{E}\{f(t; \omega)\} = \bar{f}(t)$, finite variance $\mathbb{E}\{ [f(t; \omega) - \bar{f}(t; \omega)]^2\}$ and covariance $C_f( t_1, t_2) = \mathbb{E}\{ [f(t_1; \omega) - \bar{f}(t_1; \omega)] [f(t_2; \omega) - \bar{f}(t_2; \omega)]\}$.  The KL expansion of $f(t; \omega)$ takes the form  
\begin{align}
\label{KL}
f(t;\omega) = \bar{f}(t; \omega) + \sum_{k=1}^\infty \sqrt{\lambda_k} \, \psi_k(t) \, Q_k(\omega),
\end{align}
where $\boldsymbol Q(\omega) = \lbrace Q_k(\omega) \rbrace \big|_{k=1}^{k=\infty}$ is a set of mutually uncorrelated random variables with zero mean and unit variance, while $\psi_k(t)$ and $\lambda_k$ are the eigenfunction and eigenvalues of the covariance kernel $C_f(  t_1, t_2 )$. 
We obtain the covariance kernel $C_f$ and its eigenvalues and eigenfunctions, following~\cite{Su06} and by solving a stochastic Helmholtz equation
\begin{equation}
\label{Helmholtz_roughness}
\triangle f(t; \omega) - m^2 f(t; \omega) = g( t;\omega),
\end{equation}
where the random forcing $g(t;\omega)$ is a white-noise process with zero mean and unit variance. The eigenvalues and eigenvectors of~\eqref{Helmholtz_roughness} form a Fourier series, so that the KL expansion~\eqref{KL} is replaced with its sine Fourier series version
\begin{equation}
\label{Fourier_KL}
f(t;\omega) =\bar{f}(t; \omega) +  \sum_{k=1}^\infty \, a_k \, \sin\left(\frac{2 k \pi \, t}{T}\right) \, Q_k(\omega),
\end{equation}
in which the random variables $Q_k(\omega)$ are chosen to be \textit{uniformly} distributed with probability density function $\rho_k(q_k)$. $T$ is the length of the process along the $t$-axis, and the coefficients
\begin{equation}
\label{a_k}
a_k =   \frac{2}{\sqrt T \ell^2} \left[ 1 + \left(\frac{2\pi k}{T\ell} \right)^2 \right]^{-1},
\end{equation}
where $\ell =T/A$ and $A$ is the correlation length of $f(\mathbf x;\omega)$. To ensure that the random variables $Q_k(\omega)$ have zero mean and unit variance, we define them on $Q_k(\omega) \in [-\sqrt 3, \sqrt 3]$. We note that this process is consistent to the zero-Dirichlet initial condition given in \eqref{Eq: IC}. Next, in order to render~\eqref{Fourier_KL} computable, we truncate the infinite series with a prescribed ($\approx 90\%$) fraction of the energy of the process, following the finite-dimensional noise assumption in stochastic computations. To this end, we set $T=1$, the correlation length $A=T/2$, and consider only the first four terms in the KL expansion. Let  $f_M(t;\omega) = \frac{1}{\mu} \, \sum_{k=1}^M \, a_k \, \sin\left(\frac{2 k \pi \, t}{T}\right) \, Q_k(\omega)$ denote the normalized truncated expansion, assuming $\bar{f}_M(t;\omega)=0$, where $\mu = \max_{t} \big\lbrace \sigma[f_M(t;\omega)] \big\rbrace$ and $\sigma_{f_M}$ is the standard deviation of $f_M(t;\omega)$. Thus, we represent the random process to be employed in \eqref{Eq: Stochastic FPDE} as
\begin{equation}
f(t;\omega) = \epsilon f_M(t;\omega)
\end{equation}
where $\epsilon$ is the amplitude of process. 

Therefore, the formulation of SFPDE ~\eqref{Eq: Stochastic FPDE} can be posed as follows: Find $u(t,\textbf{x};\omega): \mathbb{D}\times \Omega \rightarrow \mathbb{R}$ such that $\forall  t,\textbf{x} \in \mathbb{D}$  
\begin{align}
\label{Doob_momentum-1}
\prescript{}{0}{\mathcal{D}}_{t}^{\alpha(\omega)} u(t,\textbf{x};\omega) 
& - \sum_{j=1}^{d} \, k_{j} \,\left[ \prescript{}{a_j}{\mathcal{D}}_{x_j}^{\beta_j(\omega)}
+ \prescript{}{x_j}{\mathcal{D}}_{b_j}^{\beta_j(\omega)} \right]
u(t,\textbf{x};\omega) 
= h(t,\textbf{x}) +  f(t;Q_1(\omega) , Q_2(\omega) , \cdots , Q_M(\omega)) 
\end{align}
holds $\mathbb{P}$-a.s. for $\omega \in \Omega$, subject to the homogeneous initial and boundary conditions.

%
\subsection{Input Parametrization}
\label{Sec: Input Parametrization}
%
%
Let $Z: \Omega \rightarrow \mathbb{R}^{\mathcal{N}}$ be the set of $\mathcal{N} = 1+d+M$ independent random parameters, given as
\begin{align*}
Z = \Big\lbrace \,  Z_i \, \Big\rbrace_{i=1}^{\mathcal{N}} 
= \Big\lbrace \, 
\alpha(\omega) , \beta_1(\omega) , \beta_2(\omega) , \cdots , \beta_d(\omega) , Q_1(\omega) , Q_2(\omega) , \cdots , Q_M(\omega) 
\, \Big\rbrace
\end{align*}
with probability density functions $\rho_i : \Gamma_i \rightarrow \mathbb{R}, \, i=1,2,\cdots,\mathcal{N}$, where their images $\Gamma_i \equiv Z_i(\Omega)$ are bounded intervals in $\mathbb{R}$. The joint probability density function (PDF) 
\begin{align}
\rho(\boldsymbol \xi) = \prod_{i=1}^{\mathcal{N}} \rho_i(Z_i), \quad \forall \boldsymbol{\xi} \in \Gamma
\end{align}
with the support $\Gamma = \prod_{i=1}^{\mathcal{N}} \Gamma_{i} \subset \mathbb{R}^{\mathcal{N}}$ constitutes a mapping of the sample space $\Omega$ onto the target space $\Gamma$. Therefore, a random vector $\boldsymbol \xi = (\xi_1, \xi_2,\ldots,\xi_{\mathcal{N}}) \in \Gamma $ denote an arbitrary point in the parametric space. 

According to the Doob-Dynkin lemma~\cite{Oksendal98}, the solution $u(t,\textbf{x};\omega)$ can be expressed as $u(t,\textbf{x};\boldsymbol \xi)$, which provides a very useful tool to work in the target space rather than the abstract sample space. Thus, the formulation of SFPDE ~\eqref{Eq: Stochastic FPDE} can be posed as: Find $u(t,\textbf{x};\boldsymbol\xi): \mathbb{D}\times \Gamma \rightarrow \mathbb{R}$ such that $\forall  t,\textbf{x} \in \mathbb{D}$    
\begin{align}
\label{Doob_momentum-2}
\prescript{}{0}{\mathcal{D}}_{t}^{\alpha(\boldsymbol{\xi})} u(t,\textbf{x};\boldsymbol{\xi}) 
& - \sum_{j=1}^{d} \, k_{j} \,
\left[ \prescript{}{a_j}{\mathcal{D}}_{x_j}^{\beta_j(\boldsymbol{\xi})}
+ \prescript{}{x_j}{\mathcal{D}}_{b_j}^{\beta_j(\boldsymbol{\xi})} 
\right]
u(t,\textbf{x};\boldsymbol{\xi}) 
= h(t,\textbf{x}) +  f(t;\boldsymbol{\xi}) 
\end{align}
holds $\rho$-a.s. for $\boldsymbol\xi(\omega) \in \Gamma$ and $\forall  t,\textbf{x}\in \mathbb{D}$, subject to proper initial and boundary conditions.

%
\subsection{Stochastic Sampling}
\label{Sec: Stochastic Discretization}
%
We expound the two sampling methods, MCS and PCM to sample from random space and, then propagate the associated uncertainties by computing the statistics of stochastic solutions via post processing.

\vspace{0.1 in}
\textbf{Monte Carlo Sampling: MCS.}
The general procedure in statistical Monte Carlo sampling is the three following steps.
\begin{enumerate}
	\item Generating a set of random variables $\boldsymbol{\xi}_i$, $i = 1,2 , \cdots,K$ for a prescribed number of realizations $K$.
	
	\item Solving the deterministic problem \eqref{Doob_momentum-2} and obtaining the solution $u_i = u(t,\textbf{x};\boldsymbol{\xi}_i)$ for each $i = 1,2 , \cdots,K$.
	
	\item Computing the solution statistics, e.g. $\mathbb{E}[u] = \frac{1}{M} \sum_{i=1}^{M} u_i$.
	
\end{enumerate}
We note that step 1 and 3 are pre- and post- processing steps, respectively. Step 2 requires repetitive simulation of deterministic counterpart of the problem, which we obtain by developing a Petrov-Galerkin spectral method in the physical domain. Although MCS is relatively easy to implement once a deterministic forward solver is developed, it requires too many samplings for the solution statistics to converge, and yet the extra numerical cost due to non-locality and memory effect in fractional operators are still remained. In addition, the number of required sampling also grows rapidly as the dimension of problem increases, resulting in an exhaustively long run time for the statistics to converge.

\vspace{0.1 in}
\textbf{Probability Collocation Method: PCM.}
We employ a high-order stochastic discretization in the random space following \cite{xiu2005, Foo08} in order to construct a probabilistic collocation method (PCM), which yields a high convergence rate with much fewer number of sampling. The idea of PCM is based on polynomial interpolation, however in the random space. Let $\Theta_\mathcal{N} = \big\{ \, \xi_i \, \big\}_{i=1}^{\mathcal{J}}$ be a set of prescribed sampling points. By employing the Lagrange interpolation polynomials $L_i$, the polynomial approximation $\mathcal{I}$ of
the stochastic solution $u$ in the random space can be expressed as: 
\begin{equation}
\label{Eq: Lagrange Int}
\hat{u}(t,\textbf{x}; \boldsymbol \xi)  =
\mathcal{I}u(t,\textbf{x}; \boldsymbol \xi) 
= \sum_{i=1}^{\mathcal{J}} u(t,\textbf{x}; \xi_i)L_i(\boldsymbol{\xi}).
\end{equation}
Therefore, the collocation procedure of solving \eqref{Doob_momentum-2} to obtain the stochastic solution $u$ is:

\begin{align}
\label{Eq: PCM formulation}
R \left( \hat{u}(t,\textbf{x};\boldsymbol{\xi}) \right) \Big|_{\xi_i} 
= 
\left(
\mathcal{L}^{q(\boldsymbol{\xi})} \, \hat{u}(t,\textbf{x};\boldsymbol{\xi})) - F(t,\textbf{x};\boldsymbol{\xi}) 
\right)\Big|_{\xi_i}  = 0 \quad i =1,2,\cdots,\mathcal{J},
\end{align}
where $\mathcal{L}^{q}$ is given in \eqref{Eq: Stochastic Fractional Operator}. By using the property of Lagrange interpolants that satisfy the Kronecker delta at the interpolation points, we obtain:
\begin{align}
\label{Eq: PCM formulation-2}
\mathcal{L}^{q(\xi_i)} \, u(t,\textbf{x};\boldsymbol{\xi})) 
= F(t,\textbf{x};\boldsymbol{\xi}),
\quad i =1,2,\cdots,\mathcal{J},
\end{align}
subject to proper initial/boundary conditions. Thus, the probabilistic collocation procedure is equivalent to solving $\mathcal{J}$ deterministic problems \eqref{Eq: PCM formulation-2} with conditions \eqref{Eq: IC} and \eqref{Eq: BC}. Once the deterministic solutions are obtained at each sampling point, the numerical stochastic solution is interpolated, using \eqref{Eq: Lagrange Int}  to construct a global approximate $\hat{u} (t,\textbf{x}; \boldsymbol \xi)$. We then obtain the solution statistics as  
\begin{align}
\label{Eq: Expectation STD}
\mathbb{E}[u] = \int_{\Gamma} \hat{u} (t,\textbf{x}; \boldsymbol \xi) \, \rho(\boldsymbol \xi)  \, \mathrm d \boldsymbol \xi ,
\quad
\sigma[u] =\sqrt{\mathbb{E}[ \hat{u}^2 ] -   \mathbb{E}[\hat{u}]^2 }.
\end{align}
The above integrals can be computed efficiently by letting the interpolation/collocation points to be the same as a set of cubature rules $\Theta_\mathcal{N} = \big\{ \, \xi_i \, \big\}_{i=1}^{\mathcal{J}}$ on the parametric space with integration weights $\{\mathbf{w}_i\}_{i=1}^{\mathcal{J}}$, which are employed in computing the integral. By property of Kronecker delta of Lagrange interpolant and use of any quadrature rule over the above integral yields 
\begin{equation}
\label{Eq: Exp Approx}
\mathbb{E}[u(t,\textbf{x}:\boldsymbol{\xi})] 
\approx 
\sum_{i=1}^{\mathcal{J}} \, w_i \,u(t,\textbf{x};\xi_i).  
\end{equation}

\vspace{0.1 in}
\textbf{Choice of Collocation/Interpolation Points.}
A natural choice of the sampling points is the tensor-product of one-dimensional sets, which is efficients for low-dimensional random spaces. However, in high-dimensional multivariate case, where $\mathcal{N} > 6$, the tensor-product interpolation operators are computationally expensive due to the increasing nested summation loops. In addition, the total number of sampling points grows rapidly by increase of dimension by $\mathcal{J}^\mathcal{N}$, where $\mathcal{J}$ is the number of points in each direction.

Another choice that provides an alternative to the more costly full tensor product rule is the isotropic Smolyak sparse grid operator $A(w , \mathcal{N})$~\cite{Smolyak63,nobile2008sparse} with two input parameters dimension size $\mathcal{N}$ and the level of grid $w$. The Smolyak algorithm significantly reduces the total number of sampling points; see Fig.\ref{Fig: Grids} for comparison of $A(2,2)$, $A(4,2)$, and $A(6,2)$ with full tensor product rule for a two-dimensional random spaces. The total number of sampling points for each case is also listed in Table \ref{Table: number of sparse nodal points}. More research has also been devoted to the analysis and construction of Smolyak sparse grids \cite{xiu2005,barthelmann2000high,novak1999simple,novak1996high}.

%
%
\begin{figure}[t]
	\centering
	\begin{subfigure}{0.22\textwidth}
		\centering
		\includegraphics[width=1\linewidth]{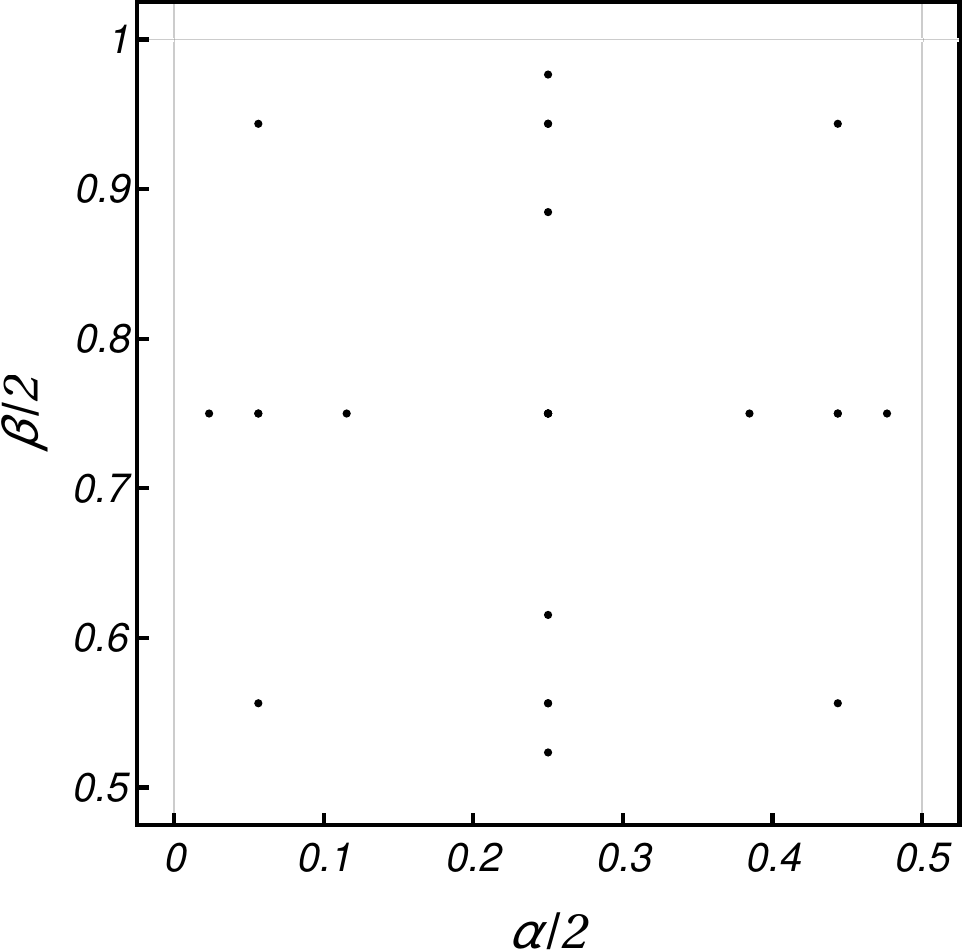}
		\caption{}
	\end{subfigure}
	\begin{subfigure}{0.22\textwidth}
		\centering
		\includegraphics[width=1\linewidth]{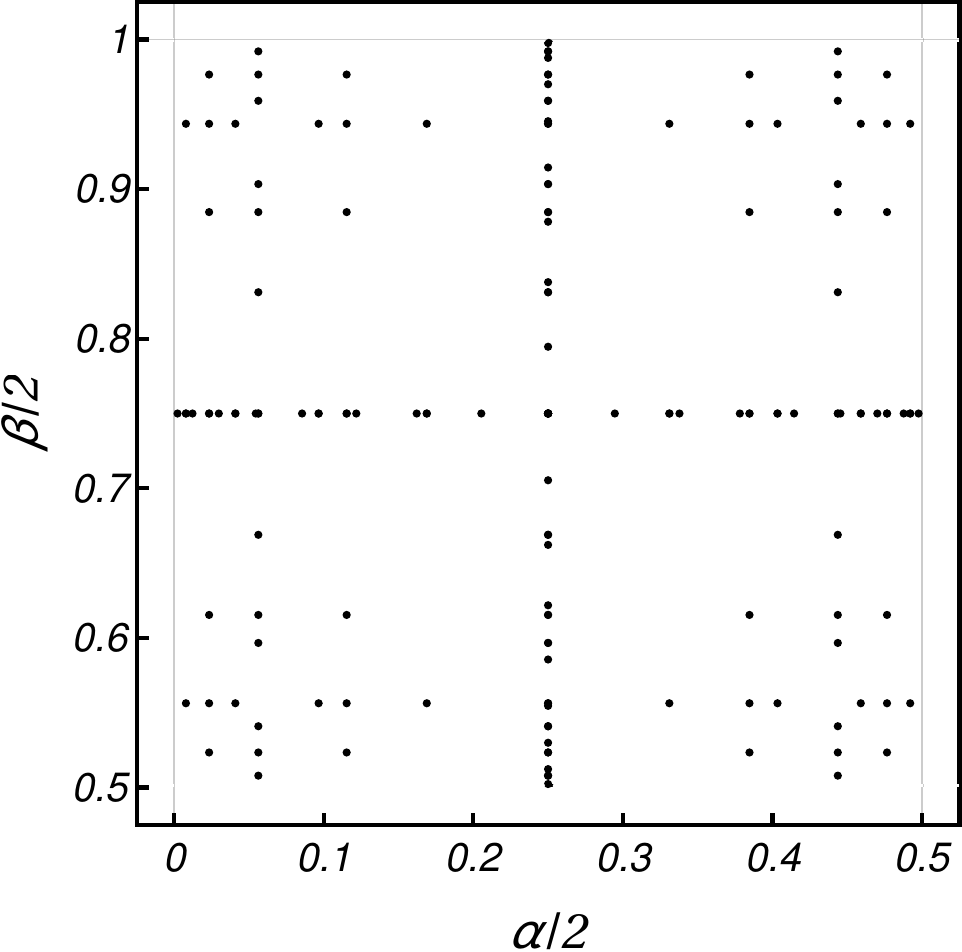}
		\caption{}
	\end{subfigure}
	\begin{subfigure}{0.22\textwidth}
		\centering
		\includegraphics[width=1\linewidth]{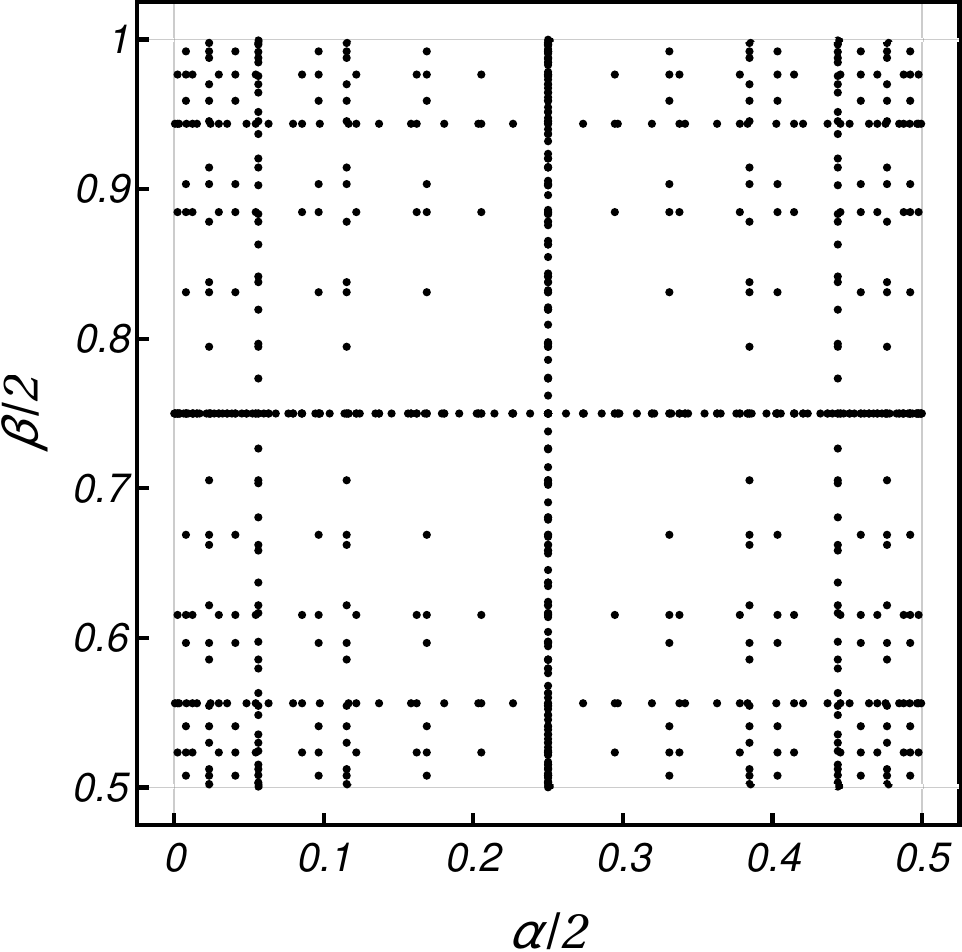}
		\caption{}
	\end{subfigure}
	\begin{subfigure}{0.22\textwidth}
		\centering
		\includegraphics[width=1\linewidth]{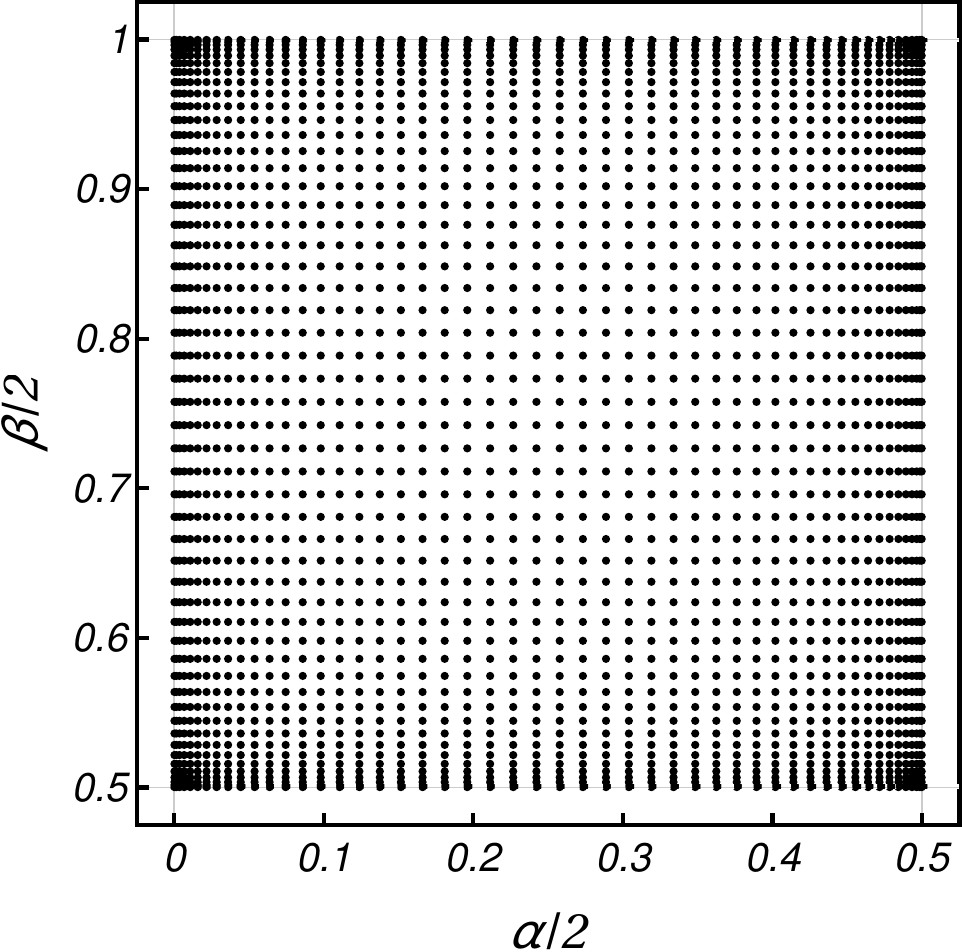}
		\caption{}
	\end{subfigure}
	\caption{Illustration of sampling nodal points in two-dimensional random space, using Smolyak sparse grid generator (a) $\, A(2,2)$, (b) $\, A(4,2)$ ,(c) $\, A(6,2)$; and (d) full tensor product rule with 50 points in each direction. The total number of points in each case is, 25, 161, 837, and 2500, respectively. }
	\label{Fig: Grids}
\end{figure}
%
%

\begin{table}[h]
	\centering
	\caption{\label{Table: number of sparse nodal points} The total number of nodal points in random space sampling, using Smolyak sparse grid generator and full tensor product with 10 points in each direction.}
	\label{my-label}
	\begin{tabular}{|c||c||c cccc}
		\hline
		Space dimensionality & Full tensor product & \multicolumn{5}{c|}{Smolyak sparse grid generator $A(w,\mathcal{N})$} \\ 
		\hline\hline
		$\mathcal{N}$ &  & \multicolumn{1}{c}{$w=2$} & \multicolumn{1}{c}{$w=4$} & \multicolumn{1}{c}{$w=6$} & \multicolumn{1}{c}{$w=8$}  & \multicolumn{1}{c|}{$w=10$}    \\ \hline
		2 & $10^2$ & \multicolumn{1}{c|}{25} & \multicolumn{1}{c|}{161} & \multicolumn{1}{c|}{837} & \multicolumn{1}{c|}{4105} & \multicolumn{1}{c|}{19469}   \\ \hline
		5 & $10^5$ & \multicolumn{1}{c|}{131} & \multicolumn{1}{c|}{3376} & \multicolumn{1}{c|}{45458} & \multicolumn{1}{c|}{440953} & \multicolumn{1}{c|}{3542465} \\ \hline
		15& $10^{15}$ & \multicolumn{1}{c|}{1066} & \multicolumn{1}{c|}{197176} & \multicolumn{1}{c|}{15480304} &  & \\ \cline{1-5}
		25& $10^{25}$ & \multicolumn{1}{c|}{2901} & \multicolumn{1}{c|}{1445975} &  &  &  \\ \cline{1-4}
		55& $10^{55}$ & \multicolumn{1}{c|}{87780} &  &  &   &  \\ \cline{1-3}
	\end{tabular}
\end{table}

%
\section{Forward Solver}
\label{Sec: Forward Solver}
%
For each realization of random variables in the employed sampling methods, the stochastic model yields a deterministic FPDE, left to be solved in the physical domain. We recall that for every $\boldsymbol{\xi}_i, \, i=1,2,\cdots$ in SFPDE \eqref{Doob_momentum-2}, the deterministic problem is recast as:
\begin{align}
\label{Eq: Deterministic FPDE}
\prescript{}{0}{\mathcal{D}}_{t}^{\alpha} u(t,\textbf{x}) 
& - \sum_{j=1}^{d} \, k_{j} \,\left[ \prescript{}{a_j}{\mathcal{D}}_{x_j}^{\beta_j}
+ \prescript{}{x_j}{\mathcal{D}}_{b_j}^{\beta_j} \right]
u(t,\textbf{x}) 
= h(t,\textbf{x}) +  f(t) ,
%
%
\end{align}
subject to the same initial/boundary conditions as \eqref{Eq: IC} and \eqref{Eq: BC}. In the sequel, we develop a Petrov-Galerkin spectral method to numerically solve the deterministic problem in the physical domain. We also show the wellposedness of deterministic problem in a weak sense and further investigate the stability of proposed numerical scheme.

%
\subsection{Mathematical Framework}
\label{Sec: Sol Test Space FPDE}
%
We define the useful functional spaces and their associated norms \cite{kharazmi2017petrov,Li2009}. By $H^\sigma(\mathbb{R}) = \big{\{} u(t) \vert u \in L^{2}(\mathbb{R});\, (1+\vert \omega \vert^2)^{\frac{\sigma}{2}} F(u)(\omega) \in L^{2}(\mathbb{R}) \big{\}}$, $\sigma \geq 0$, we denote the fractional Sobolev space on $\mathbb{R}$, endowed with norm $\Vert u \Vert_{H^{\sigma}_{\mathbb{R}}}=\Vert (1+\vert \omega \vert^2)^{\frac{\sigma}{2}} F(u)(\omega) \Vert_{L^{2}(\mathbb{R})}$, where $\mathcal{F}(u)$ represents the Fourier transform of $u$. Subsequently, we denote by $H^\sigma(\Lambda) =\big{\{} u\in L^{2}(\Lambda)\, \vert \,\exists \, \tilde{u} \in H^{\sigma}(\mathbb{R})\, \, s.t. \, \,\tilde{u}\vert_{\Lambda}=u \big{\}}$, $\sigma \geq 0$, the fractional Sobolev space on any finite closed interval, e.g. $\Lambda = (a,b)$, with norm $ \Vert u \Vert_{H^{\sigma}(\Lambda)}= \underset{\tilde{u}\in H^{\sigma}_{\mathbb{R}},\, \tilde{u}\vert_{\Lambda}=u }{\inf} \, \Vert \tilde{u} \Vert_{{}H^{\sigma}(\mathbb{R})}$. We define the following useful norms as:
\begin{align*}
&\Vert \cdot \Vert_{{^l}H^{\sigma}_{}(\Lambda)} = \Big(\Vert \prescript{}{a}{\mathcal{D}}_{x}^{\sigma}\, (\cdot)\Vert_{L^2(\Lambda)}^2+\Vert \cdot \Vert_{L^2(\Lambda)}^2 \Big)^{\frac{1}{2}},
\\
&\Vert \cdot \Vert_{{^r}H^{\sigma}_{}(\Lambda)} = \Big(\Vert \prescript{}{x}{\mathcal{D}}_{b}^{\sigma}\, (\cdot)\Vert_{L^2(\Lambda)}^2+\Vert \cdot \Vert_{L^2(\Lambda)}^2 \Big)^{\frac{1}{2}},
\\
&\Vert \cdot \Vert_{{^c}H^{\sigma}_{}(\Lambda)} = \Big(\Vert \prescript{}{x}{\mathcal{D}}_{b}^{\sigma}\, (\cdot)\Vert_{L^2(\Lambda)}^2+\Vert \prescript{}{a}{\mathcal{D}}_{x}^{\sigma}\, (\cdot)\Vert_{L^2(\Lambda)}^2+\Vert \cdot \Vert_{L^2(\Lambda)}^2 \Big)^{\frac{1}{2}},
\end{align*}
where the equivalence of $\Vert \cdot \Vert_{{^l}H^{\sigma}_{}(\Lambda)}$ and $\Vert \cdot \Vert_{{^r}H^{\sigma}_{}(\Lambda)}$ are shown in \cite{Li2009,ervin2007variational,Li2010}. 
\begin{lem}
	\label{Lem: norm equivalence 1}
	Let $\sigma \geq 0$ and $\sigma \neq n-\frac{1}{2}$. Then, the norms $\Vert \cdot \Vert_{{^l}H^{\sigma}_{}(\Lambda)}$ and $\Vert \cdot \Vert_{{^r}H^{\sigma}_{}(\Lambda)}$ are equivalent to $\Vert \cdot \Vert_{{^c}H^{\sigma}_{}(\Lambda)}$.
\end{lem}

\noindent We also define $C^{\infty}_{0}(\Lambda)$ as the space of smooth functions with compact support in $(a,b)$. We denote by $\prescript{l}{}H^{\sigma}_0(\Lambda)$, $\prescript{r}{}H^{\sigma}_0(\Lambda)$, and $\prescript{c}{}H^{\sigma}_0(\Lambda)$ as the closure of $C^{\infty}_{0}(\Lambda)$ with respect to the norms $\Vert \cdot \Vert_{{^l}H^{\sigma}_{}(\Lambda)}$, $\Vert \cdot \Vert_{{^r}H^{\sigma}_{}(\Lambda)}$, and $\Vert \cdot \Vert_{{^c}H^{\sigma}_{}(\Lambda)}$. It is shown in \cite{ervin2007variational,Li2010} that these Sobolev spaces are equal and their seminorms are also equivalent to $\vert \cdot \vert_{{}H^{\sigma}_{}(\Lambda)}^{*} = \big\vert \left( \prescript{}{a}{\mathcal{D}}_{x}^{\sigma}\, (\cdot), \prescript{}{x}{\mathcal{D}}_{b}^{\sigma}\, (\cdot) \right) \big\vert_{\Lambda}^{\frac{1}{2}}$. Therefore, we can prove that $\big{\vert}(\prescript{}{a}{\mathcal{D}}_{x}^{\sigma}\, u, \prescript{}{x}{\mathcal{D}}_{b}^{\sigma}\, v )_{\Lambda}^{} \big{\vert} \geq \beta \, \vert u \vert_{{^l}H^{\sigma}_{}(\Lambda)}\, \vert  v \vert_{{^r}H^{\sigma}_{}(\Lambda)}$ and $\big{\vert}(\prescript{}{x}{\mathcal{D}}_{b}^{\sigma}\, u, \prescript{}{a}{\mathcal{D}}_{x}^{\sigma}\, v )_{\Lambda}^{} \big{\vert} \geq \beta \, \vert u \vert_{{^r}H^{\sigma}_{}(\Lambda)}\, \vert  v \vert_{{^l}H^{\sigma}_{}(\Lambda)}$, in which $\beta$ is a positive constant.

Moreover, by letting ${_0}C^{\infty}(I)$ and $C^{\infty}_{0}(I)$ be the space of smooth functions with compact support in $(0,T]$ and $[0,T)$, respectively, we define $\prescript{l}{}H^{s}(I)$ and $\prescript{r}{}H^{s}(I)$ as the closure of ${_0}C^{\infty}(I)$ and $C^{\infty}_{0}(I)$ with respect to the norms $\Vert \cdot \Vert_{\prescript{l}{}H^{s}(I)}$ and $\Vert \cdot \Vert_{\prescript{r}{}H^{s}(I)}$. Other equivalent useful semi-norms associated with $H^s(I)$ are also introduced in \cite{Li2009,ervin2007variational}, as $\vert \cdot \vert_{{^l}H^{s}(I)} = \Vert \prescript{}{0}{\mathcal{D}}_{t}^{s} (\cdot)\Vert_{L^2(I)}$, $ \vert \cdot \vert_{{^r}H^{s}(I)} = \Vert \prescript{}{t}{\mathcal{D}}_{T}^{s} (\cdot) \Vert_{L^2(I)}$, $\vert \cdot \vert_{{}H^{s}(I)}^* = \big\vert \left( \prescript{}{0}{\mathcal{D}}_{t}^{s}(\cdot),\prescript{}{t}{\mathcal{D}}_{T}^{s}(\cdot) \right)_{I} \big\vert^{\frac{1}{2}}$, where $\vert \cdot \vert_{{}H^{s}(I)}^* \equiv \vert \cdot \vert_{{^l}H^{s}(I)}^{\frac{1}{2}} \, \vert \cdot \vert_{{^r}H^{s}(I)}^{\frac{1}{2}}$.

Borrowing definitions from \cite{samiee2016}, we define the following spaces, which we use later in construction of corresponding solution and test spaces of our problem. Thus, by letting $\Lambda_1 = (a_1,b_1)$, $\Lambda_j = (a_j,b_j) \times \Lambda_{j-1}$ for $j=2,\cdots,d$, we define $\mathcal{X}_1 = H^{\frac{\beta_1}{2}}_{0}(\Lambda_1)$, which is associated with the norm $ \Vert \cdot \Vert_{{^c}H^{\frac{\beta_1}{2}}_{}(\Lambda_1)}$, and accordingly, $\mathcal{X}_j, \, j=2,\cdots,d$ as
\begin{eqnarray}
\mathcal{X}_2 &=& H^{\frac{\beta_2}{2}}_0 \Big((a_2,b_2); L^2(\Lambda_1) \Big) \cap L^2((a_2,b_2); \mathcal{X}_1),
\\
&\vdots&
\nonumber
\\
\mathcal{X}_d &=& H^{\frac{\beta_d}{2}}_0 \Big((a_d,b_d); L^2(\Lambda_{d-1}) \Big) \cap L^2((a_d,b_d); \mathcal{X}_{d-1}),
\end{eqnarray}
associated with norms $\Vert \cdot \Vert_{\mathcal{X}_j} = \bigg{\{} \Vert \cdot \Vert_{H^{\frac{\beta_j}{2}}_0 \Big((a_j,b_j); L^2(\Lambda_{j-1}) \Big)}^2 + \Vert \cdot \Vert_{ L^2\Big((a_j,b_j); \mathcal{X}_{j-1}\Big)}^2 \bigg{\}}^{\frac{1}{2}}, \,\, j=2,3,\cdots,d$.

\vspace{0.1 in}
\begin{lem}
	\label{space norm 1}
	Let $\sigma \geq 0$ and $\sigma \neq n-\frac{1}{2}$. Then, for $j=1,2,\cdots,d$ 
	\begin{align*}
	%
	\Vert \cdot \Vert_{\mathcal{X}_j} \equiv \bigg{\{}  \sum_{i=1}^{j} \Big(\Vert \prescript{}{x_i}{\mathcal{D}}_{b_i}^{\beta_i/2}\, (\cdot)\Vert_{L^2(\Lambda_j)}^2+\Vert \prescript{}{a_i}{\mathcal{D}}_{x_i}^{\beta_i/2}\, (\cdot)\Vert_{L^2(\Lambda_j)}^2 \Big) + \Vert  \cdot \Vert_{L^2(\Lambda_j)}^2 \bigg{\}}^{\frac{1}{2}}.
	\end{align*}
\end{lem}
%

%
\subsection*{\textbf{Solution and Test Spaces}}
%
We define the ``solution space" $U$ and ``test space" $V$, respectively, as
\begin{align}
\label{Eq: solution test space}
U  = \prescript{l}{0}H^{\frac{\alpha}{2}}\Big(I; L^2(\Lambda_d) \Big) \cap L^2(I; \mathcal{X}_d),
\quad%
V  = \prescript{r}{0}H^{\frac{\alpha}{2}}\Big(I; L^2(\Lambda_d)\Big) \cap L^2(I; \mathcal{X}_d),
\end{align}
endowed with norms
\begin{align}
\label{Eq: solution test space norm}
\Vert u \Vert_{U} &= 
\Big{\{}\Vert u \Vert_{\prescript{l}{}H^{\frac{\alpha}{2}}(I; L^2(\Lambda_d))}^2 + \Vert u \Vert_{L^2(I; \mathcal{X}_d)}^2 \Big{\}}^{\frac{1}{2}}, 
\quad
\Vert v \Vert_{V} &= 
\Big{\{}\Vert v \Vert_{\prescript{r}{}H^{\tau}(I; L^2(\Lambda_d))}^2  + \Vert v \Vert_{ L^2(I; \mathcal{X}_d)}^2 \Big{\}}^{\frac{1}{2}},
\end{align}
where $I = [0,T]$, and 
\begin{align*}
\prescript{l}{0}H^{\frac{\alpha}{2}} \Big(I; L^2(\Lambda_d) \Big) = 
\Big{\{} u \, \big|\, \Vert u(t,\cdot) \Vert_{L^2(\Lambda_d)} \in H^{\frac{\alpha}{2}}(I), u\vert_{t=0}=u\vert_{x=a_j}=u\vert_{x=b_j}=0,\, j=1,2,\cdots,d  \Big{\}},
\\
\prescript{r}{0}H^{\frac{\alpha}{2}} \Big(I; L^2(\Lambda_d) \Big) = 
\Big{\{} v \,\big|\, \Vert v(t,\cdot) \Vert_{L^2(\Lambda_d)} \in H^{\frac{\alpha}{2}}(I), v\vert_{t=T}=v\vert_{x=a_j}=v\vert_{x=b_j}=0,\, j =1,2,\cdots,d  \Big{\}},
\end{align*}
equipped with norms $\Vert u \Vert_{\prescript{l}{}H^{\frac{\alpha}{2}}(I; L^2(\Lambda_d))}$ and $\Vert u \Vert_{\prescript{r}{}H^{\frac{\alpha}{2}}(I; L^2(\Lambda_d))}$, respectively. We can show that these norms take the following forms
\begin{align}
%
\label{space norm 2}
&
\Vert u \Vert_{\prescript{l}{}H^{\frac{\alpha}{2}}(I; L^2(\Lambda_d))} = \Big{\Vert} \, \Vert u(t,\cdot) \Vert_{L^2(\Lambda_d)}\, \Big{\Vert}_{{^l}H^{\frac{\alpha}{2}}(I)}
=\Big(\Vert \prescript{}{0}{\mathcal{D}}_{t}^{\frac{\alpha}{2}}\, (u)\Vert_{L^2(\Omega)}^2 + \Vert u\Vert_{L^2(\Omega)}^2 \Big)^{\frac{1}{2}},
\\
&\Vert u \Vert_{\prescript{r}{}H^{\frac{\alpha}{2}}(I; L^2(\Lambda_d))} = \Big{\Vert} \, \Vert u(t,\cdot) \Vert_{L^2(\Lambda_d)}\, \Big{\Vert}_{{^r}H^{\frac{\alpha}{2}}(I)}
= \Big(\Vert \prescript{}{t}{\mathcal{D}}_{T}^{\frac{\alpha}{2}}\, (u)\Vert_{L^2(\Omega)}^2+\Vert u\Vert_{L^2(\Omega)}^2\Big)^{\frac{1}{2}}.
\end{align}
Also, using Lemma \ref{space norm 1}, we can show that
\begin{align}
%
\label{space norm 3}
\Vert u \Vert_{L^2(I; \mathcal{X}_d)}
=
\Big{\Vert} \, \Vert u(t,.) \Vert_{\mathcal{X}_d}\,\Big{\Vert}_{L^2(I)}
=
\Big{\{}  \Vert u \Vert_{L^2(\Omega)}^2 + \sum_{j=1}^{d} \big( \Vert \prescript{}{x_j}{\mathcal{D}}_{b_j}^{\frac{\beta_j}{2}}\, (u)\Vert_{L^2(\Omega)}^2 
+ \Vert \prescript{}{a_j}{\mathcal{D}}_{x_j}^{\frac{\beta_j}{2}}\, (u)\Vert_{L^2(\Omega)}^2 \big) 
\Big{\}}^{\frac{1}{2}}.
\end{align}
Therefore, \eqref{Eq: solution test space norm} can be written as
\begin{align}
\label{Eq: solution space norm}
\Vert u \Vert_{U} 
&= 
\Big{\{}  \Vert u \Vert_{L^2(\Omega)}^2 + \Vert \prescript{}{0}{\mathcal{D}}_{t}^{\frac{\alpha}{2}}\, (u)\Vert_{L^2(\Omega)}^2 
+ \sum_{j=1}^{d} \big( \Vert \prescript{}{x_j}{\mathcal{D}}_{b_j}^{\frac{\beta_j}{2}}\, (u)\Vert_{L^2(\Omega)}^2+\Vert \prescript{}{a_j}{\mathcal{D}}_{x_j}^{\frac{\beta_j}{2}}\, (u)\Vert_{L^2(\Omega)}^2 \big) \Big{\}}^{\frac{1}{2}}, 
\\
\label{Eq: test space norm}
\Vert v \Vert_{V} 
& =
\Big{\{}  \Vert v \Vert_{L^2(\Omega)}^2 + \Vert \prescript{}{t}{\mathcal{D}}_{T}^{\frac{\alpha}{2}}\, (v)\Vert_{L^2(\Omega)}^2 
+ \sum_{j=1}^{d} \big( \Vert \prescript{}{x_j}{\mathcal{D}}_{b_j}^{\frac{\beta_j}{2}}\, (v)\Vert_{L^2(\Omega)}^2
+\Vert \prescript{}{a_j}{\mathcal{D}}_{x_j}^{\frac{\beta_j}{2}}\, (v)\Vert_{L^2(\Omega)}^2 \big) \Big{\}}^{\frac{1}{2}}.
\end{align}
%

%
%
\subsection{Weak Formulation}
\label{Sec: weak formulation}
%
%
The following lemmas help us obtain the weak formulation of deterministic problem in the physical domain and construct the numerical scheme.

%
\begin{lem}
	\label{Lem: left frac proj}
	\cite{Li2009}: For all $\alpha \in  (0,1)$, if $u \in H^1([0,T])$ such that $u(0)=0$, and $v \in H^{\alpha/2}([0,T])$, then $( \prescript{}{0}{ \mathcal{D}}_{t}^{\,\,\alpha} u, v )_{\Omega} =  (\, \prescript{}{0}{ \mathcal{D}}_{t}^{\,\,\alpha/2} u \,,\, \prescript{}{t}{ \mathcal{D}}_{T}^{\,\,\alpha/2} v\, )_{\Omega}$, where $(\cdot , \cdot)_{\Omega}$ represents the standard inner product in $\Omega=[0,T]$. 
\end{lem}
%
%
\begin{lem}
	\label{Lem: fractional integ-by-parts 1 and 2}
	\cite{kharazmi2017petrov}: Let $1 < \beta < 2$, $a$ and $b$ be arbitrary finite or infinite real numbers. Assume $u \in H^{\beta}(a,b)$ such that $u(a)=0$, also $\prescript{}{x}{\mathcal{D}}_{b}^{\beta/2}v$ is integrable in $(a,b)$ such that $v(b) = 0$. Then, $( \prescript{}{a}{\mathcal{D}}_{x}^{\beta} u \,,\,v ) = ( \prescript{}{a}{\mathcal{D}}_{x}^{\beta/2} u \,,\,\prescript{}{x}{\mathcal{D}}_{b}^{\beta/2} v )$.
\end{lem}
\begin{lem}
	\label{lem_generalize}
	Let $1<\beta_j<2$ for $j=1,2,\cdots,d$, and $u,v \in  \mathcal{X}_d$. Then,  
	\begin{align*}
	\big(\prescript{}{a_j}{\mathcal{D}}_{x_j}^{\beta_j} u, v\big)_{\Lambda_d}=\big(\prescript{}{a_j}{\mathcal{D}}_{x_j}^{\frac{\beta_j}{2}} u, \prescript{}{x_j}{\mathcal{D}}_{b_j}^{\frac{\beta_j}{2}} v\big)_{\Lambda_d},
	\qquad
	\big(\prescript{}{x_j}{\mathcal{D}}_{b_j}^{\beta_j} u, v\big)_{\Lambda_d}=\big(\prescript{}{x_j}{\mathcal{D}}_{b_j}^{\frac{\beta_j}{2}} u, \prescript{}{a_j}{\mathcal{D}}_{x_j}^{\frac{\beta_j}{2}} v\big)_{\Lambda_d}.
	\end{align*}
	%
	%
\end{lem}

For any realization of \eqref{Doob_momentum-2}, we obtain the weak system, i.e. the variational form of the deterministic counterpart of the problem problem, subject to the given initial/boundary conditions, by multiplying the equation with proper test functions and integrate over the whole computational domain $\mathbb{D}$. Using Lemmas \ref{Lem: left frac proj}-\ref{lem_generalize}, the bilinear form can be written as
\begin{align}
%
\label{Eq: bilinear form}
a(u,v)
=(\prescript{}{0}{\mathcal{D}}_{t}^{\frac{\alpha}{2}}\, u, \prescript{}{t}{\mathcal{D}}_{T}^{\frac{\alpha}{2}}\, v )_{\mathbb{D}} 
-\sum_{j=1}^{d} 
k_{j} \Big[ ( \prescript{}{a_j}{\mathcal{D}}_{x_j}^{\frac{\beta_j}{2}}\, u,\, \prescript{}{x_j}{\mathcal{D}}_{b_j}^{\frac{\beta_j}{2}}\, v )_{\mathbb{D}}
+ ( \prescript{}{x_j}{\mathcal{D}}_{b_j}^{\frac{\beta_j}{2}}\, u , \, \prescript{}{a_j}{\mathcal{D}}_{x_j}^{\frac{\beta_j}{2}} v)_{\mathbb{D}}
\Big] ,
\end{align}
and thus, by letting $U$ and $V$ be the proper solution/test spaces, the problem reads as: find $u \in U$ such that
\begin{align}
\label{Eq: general weak form FPDE}
a(u,v) = (\text{f},v)_{\mathbb{D}}, \quad \forall v \in V ,
\end{align}
where $\text{f} =  h(t,\textbf{x}) +  f(t) $.

%
%
\subsection{Petrov-Galerkin Spectral Method}
\label{Sec: Implementation}
%
%
We define the following finite dimensional solution and test spaces. We employ Legendre polynomials $\phi_{m_j}(\xi), \, j=1,2,\cdots,d$, and Jacobi poly-fractonomial of first kind $\psi^{\tau}_n(\eta)$ \cite{zayernouri2015tempered,Zayernouri2013}, as the spatial and temporal bases, respectively, given in their corresponding standard domain as
\begin{align}
\label{Eq: Spatial Basis}
\phi^{}_{m_j} ( \xi )  & =  \sigma_{m_j} \big{(} P_{m_j+1} (\xi) - P_{m_j-1} (\xi)\big{)},  \quad  \xi \in [-1,1]  \qquad m_j=1,2,\cdots ,
\\
\label{Eq: Temporal Basis}
\psi^{\tau}_n(\eta) & = {\sigma}_{n} \prescript{(1)}{}{ \mathcal{P}}_{n}^{\,\,\tau}(\eta) = {\sigma}_{n} (1+\eta)^{\tau} P_{n-1}^{-\tau, \tau} (\eta), \quad \eta\in [-1,1]  \quad n=1,2,\cdots ,
\end{align}
in which $\sigma_{m_j} = 2 + (-1)^{m_j}$. Therefore, by performing affine mappings $\eta = 2\frac{t}{T}-1$ and $\xi = 2\frac{x-a_j}{b_j-a_j} -1$ from the computational domain to the standard domain, we construct the solution space $U_N$ as
\begin{align}
\label{Eq: Solution Space :PG}
U_N = 
span \, \Big\{ \,\,   
\Big( \psi^{\,\tau}_n \circ \eta \Big) ( t ) \,\,
\prod_{j=1}^{d} \Big( \phi^{}_{m_j} \circ \xi \Big)  (x_j) \,\,
: n = 1,2, \cdots, \mathcal{N}, \,\, m_j= 1,2, \cdots, \mathcal{M}_j
\,\, \Big\}.
\end{align}
We note that the choice of temporal and spatial basis functions naturally satisfy the initial and boundary conditions, respectively. The parameter $\tau$ in the temporal basis functions plays a role of fine tunning parameter, which can be chosen properly to capture the singularity of exact solution.

Moreover, we employ Legendre polynomials $\Phi_{r_j}(\xi), \, j=1,2,\cdots,d$, and Jacobi poly-fractonomial of second kind $\Psi^{\tau}_k(\eta)$, as the spatial and temporal test functions, respectively, given in their corresponding standard domain as
\begin{align}
\label{Eq: Spatial Test}
\Phi_{r_j} ( \xi )  & =  \widetilde{\sigma}_{r_j} \big{(} P_{r_j+1}^{} (\xi) - P_{r_j-1}^{} (\xi)\big{)},  \quad  \xi \in [-1,1]  \qquad r_j =1,2,\cdots ,
\\
\label{Eq: Temporal Test}
\Psi^{\tau}_k(\eta) & = \widetilde{\sigma}_{k} \prescript{(2)}{}{ \mathcal{P}}_{k}^{\,\,\tau}(\eta) = \widetilde{\sigma}_{k} (1-\eta)^{\tau}\, P_{k-1}^{\tau,-\tau} (\eta), \quad \eta\in [-1,1]  \quad k=1,2,\cdots,
\end{align}
where $ \widetilde{\sigma}_{r_j} = 2\,(-1)^{r_j} + 1$. Therefore, by similar affine mapping we construct the test space $V_N$ as
\begin{align}
\label{Eq: Test Space: PG}
V_N = span \, \Big\{  \,\,
\Big(\Psi^{\tau}_k \circ \eta\Big)(t) \,\,
\prod_{j=1}^{d} \Big( \Phi^{}_{r_j} \circ \xi_j\Big)(x_j) \,\,
: k = 1,2, \cdots, \mathcal{N}, \,\, r_j= 1,2, \cdots, \mathcal{M}_j
\,\, \Big\}.
\end{align}
Thus, since $U_N \subset U$ and $V_N \subset  V$, the problems \eqref{Eq: general weak form FPDE} read as: find $u_N \in U_N$ such that
\begin{align}
\label{Eq: PG method FPDE}
a_h(u_N,v_N) = l(v_N), \quad \forall v_N \in V_N,
\end{align}
where $l(v_N) = (\text{f},v_N)$. The discrete bilinear form $a_h(u_N,v_N)$ can be written as 
\begin{align}
\label{Eq: discrete weak form}
a_h(u_N,v_N)
=(\prescript{}{0}{\mathcal{D}}_{t}^{\frac{\alpha}{2}}\, u_N, \prescript{}{t}{\mathcal{D}}_{T}^{\frac{\alpha}{2}}\, v_N )_{\mathbb{D}} 
-\sum_{j=1}^{d} 
k_{j} \Big[ ( \prescript{}{a_j}{\mathcal{D}}_{x_j}^{\frac{\beta_j}{2}}\, u_N ,\, \prescript{}{x_j}{\mathcal{D}}_{b_j}^{\frac{\beta_j}{2}}\, v_N )_{\mathbb{D}}
+ ( \prescript{}{x_j}{\mathcal{D}}_{b_j}^{\frac{\beta_j}{2}}\, u_N , \, \prescript{}{a_j}{\mathcal{D}}_{x_j}^{\frac{\beta_j}{2}} v_N)_{\mathbb{D}}
\Big].
\end{align}
We expand the approximate solution $u_N \in U_N$, satisfying the discrete bilinear form \eqref{Eq: discrete weak form}, in the following form
\begin{align}
\label{Eq: PG expansion}
u_{N}(t,\textbf{x}) = 
\sum_{n=1}^\mathcal{N}
\sum_{m_1=1}^{\mathcal{M}_1}
\cdots 
\sum_{m_d= 1}^{\mathcal{M}_d}  \,\,
\hat u_{ n,m_1,\cdots,m_d} \,\,
\Big[\psi^{\tau}_n(t)
\prod_{j=1}^{d} \phi^{}_{m_j}(x_j)
\Big] ,
\end{align}
and obtain the corresponding Lyapunov system by substituting \eqref{Eq: PG expansion} into \eqref{Eq: discrete weak form} by choosing $v_N(t,\textbf{x}) = \Psi^{\tau}_k(t) \prod_{j=1}^{d} \Phi^{}_{r_j}(x_j)$, $k = 1,2, \dots, \mathcal{N}$, $r_j= 1,2, \dots, \mathcal{M}_j$. Therefore, 
\begin{align}
\label{Eq: general Lyapunov}
\Big[
S_{T} \otimes M_1 \otimes M_2 \cdots \otimes M_d 
&+
\sum_{j=1}^{d} 
M_{T} \otimes M_1\otimes \cdots   \otimes M_{j-1} \otimes S_{j}^{{Tot}} \otimes M_{j+1}  \cdots \otimes M_d
\nonumber
\\
&+ 
\gamma \, M_{T}\otimes M_1 \otimes M_2 \cdots \otimes M_d 
\Big] \, 
\mathcal{U}= F,
\end{align} 
in which $\otimes$ represents the Kronecker product, $F$ denotes the multi-dimensional load matrix whose entries are given as
\begin{eqnarray}
\label{Eq: general load matrix}
F_{k,r_1,\cdots, r_d} = \int_{\mathbb{D}}^{} \text{f}(t,\textbf{x}) \,
\Big(
\Psi^{\,\tau}_k \circ \eta \Big)(t)
\prod_{j=1}^{d} \Big(\Phi^{}_{r_j} \circ \xi_j\Big)(x_j)\, 
d\mathbb{D},
\end{eqnarray}
and $\mathcal{U}$ is the matrix of unknown coefficients. The matrices $S_{T}$ and $M_{T}$ denote the temporal stiffness and mass matrices, respectively; and the matrices $S_{j}$ and $M_j$ denote the spatial stiffness and mass matrices, respectively. We obtain the entries of spatial mass matrix $M_j$ analytically and employ proper quadrature rules to accurately compute the entries of other matrices $S_{T}$, $M_{T}$ and $S_{j}$.

We note that the choices of basis/test functions, employed in developing the PG scheme leads to symmetric mass and stiffness matrices, providing useful properties to further develop a fast solver. The following Theorem \ref{Thm: fast solver} provides a unified fast solver, developed in terms of the generalized eigensolutions in order to obtain a closed-form solution to the Lyapunov system \eqref{Eq: general Lyapunov}.

\begin{thm}[Unified Fast FPDE Solver \cite{samiee2016}]
	\label{Thm: fast solver}
	Let $\{ {\vec{e}_{}}^{\,\,\mu_j}    ,    \lambda^{}_{m_j}\,  \}_{m_j=1}^{\mathcal{M}_j}$ be the set of general eigen-solutions of the spatial stiffness matrix $S^{Tot}_j$ with respect to the mass matrix $M_{j}$. Moreover, let $\{ {\vec{e}_{n}}^{\,\,\tau}    ,    \lambda^{\tau}_{n}\,  \}_{n=1}^{\mathcal{N}}$ be the set of general eigen-solutions of the temporal mass matrix $M_{T}$ with respect to the stiffness matrix $S_{T}$. Then, the matrix of unknown coefficients $\mathcal{U}$ is explicitly obtained as
	\begin{equation}
	\label{Eq: thm u expression in terms of k}
	\mathcal{U} = 
	\sum_{n=1}^{\mathcal{N}}
	\,\,
	\sum_{m_1= 1}^{\mathcal{M}_1}
	\cdots 
	\sum_{m_d= 1}^{\mathcal{M}_d}
	\kappa_{ n,m_1,\cdots,\,m_d  } \,
	\,\vec{e}_n^{\,\,\tau}\,
	\otimes
	\,{\vec{e}_{m_1}}^{}\,\,
	\otimes
	\cdots
	\otimes
	\,{\vec{e}_{m_d}}^{},
	\end{equation}
	where $\kappa_{ n,m_1,\cdots,\,m_d }$ is given by 
	\begin{eqnarray}
	\label{Eq: thm k fraction_1}
	\kappa_{ n,m_1,\cdots,\,m_d  } =  \frac{(\,\vec{e}_n^{\,\,\tau}
		\,{\vec{e}_{m_1}}^{}
		\cdots
		\,{\vec{e}_{m_d}}^{}) F}
	{
		\Big[
		(\vec{e}_n^{\,\,\tau^T} \, S_{T} \, \vec{e}_n^{\,\,\tau}) \,
		\prod_{j=1}^{d} (\vec{e}_{m_j}^{T}  \,  M_{j} \,  {\vec{e}_{m_j}}^{}) \,
		\Big]
		\Lambda_{n,m_1,\cdots,m_d}
	},
	\end{eqnarray}
	in which the numerator represents the standard multi-dimensional inner product, and $\Lambda_{n,m_1,\cdots,m_d}$ is obtained in terms of the eigenvalues of all mass matrices as
	\begin{eqnarray}
	\nonumber
	&\Lambda_{n,m_1,\cdots, m_d} = \Big[
	(1+\gamma\,\, 
	\lambda^{\tau}_n)
	+
	\lambda^{\tau}_n
	\sum_{j=1}^{d}
	(
	\lambda^{}_{m_j}
	)
	\Big].  &
	\end{eqnarray}
	
\end{thm}

%
%
%
%
%
\subsection{Stability Analysis}
\label{Sec: Stability and Convergence of PG}
%
%

We show the well-posedness of deterministic problem and prove the stability of proposed PG scheme. 
\begin{lem}
	\label{norm_223}
	Let $\alpha \in (0,1)$, $\Omega=I \times \Lambda_d$, and $u\in \prescript{l}{0}H^{\alpha/2}(I; L^2(\Lambda_d))$. Then, 
	\begin{equation*}
	\big\vert \left( \prescript{}{0}{\mathcal{D}}_{t}^{\alpha/2} u, \prescript{}{t}{\mathcal{D}}_{T}^{\alpha/2} v \right)_{\Omega} \big\vert \equiv \Vert u \Vert_{\prescript{l}{}H^{\alpha/2}(I; L^2(\Lambda_d))} \, \Vert v \Vert_{\prescript{r}{}H^{\alpha/2}(I; L^2(\Lambda_d))},
	\quad 
	\forall v \in \prescript{r}{0}H^{\alpha/2}(I; L^2(\Lambda_d)).
	\end{equation*}
\end{lem}
Moreover,
\begin{align}
\label{equiv_space}
\vert \big(\prescript{}{a_d}{\mathcal{D}}_{x_d}^{\beta_d/2} u, \prescript{}{x_d}{\mathcal{D}}_{b_d}^{\beta_d/2} v\big)_{\Lambda_d} \vert \equiv  \vert u \vert_{\prescript{c}{}H^{\beta_d/2} \Big((a_d,b_d); L^2(\Lambda_{d-1}) \Big)} \, \vert v \vert_{\prescript{c}{}H^{\beta_d/2} \Big((a_d,b_d); L^2(\Lambda_{d-1}) \Big)},
\\
\label{equiv_space2}
\vert \big(\prescript{}{x_d}{\mathcal{D}}_{b_d}^{\beta_d/2} u, \prescript{}{a_d}{\mathcal{D}}_{x_d}^{\beta_d/2} v\big)_{\Lambda_d} \vert \equiv  \vert u \vert_{\prescript{c}{}H^{\beta_d/2} \Big((a_d,b_d); L^2(\Lambda_{d-1}) \Big)} \, \vert v \vert_{\prescript{c}{}H^{\beta_d/2} \Big((a_d,b_d); L^2(\Lambda_{d-1}) \Big)}.
\end{align} 
\begin{lem}[Continuity]
	\label{continuity_lem}
	The bilinear form \eqref{Eq: bilinear form} is continuous, i.e.,
	\begin{align}
	\label{continuity_eq}
	\forall u \in U, \,\,
	\exists \, \beta > 0,  
	\quad \text{s.t.} \quad 
	\vert a(u,v)\vert \leq 
	\beta \,\,   \Vert u \Vert_{U}   \,\,     \Vert v \Vert_{V},
	\quad \forall v \in V.
	\end{align}
\end{lem}
\begin{proof}
	The proof directly concludes from \eqref{equiv_space} and Lemma \ref{norm_223}.
\end{proof}	

\begin{thm}[Stability]
	\label{inf_sup_d_lem}
	The following inf-sup condition holds for the bilinear form \eqref{Eq: bilinear form}, i.e.,
	\begin{align}
	\label{Eq: inf sup-time_d_well}
	\underset{ u \neq 0  \in U}{\inf} \,\, \underset{ v \neq 0 \in V}{\sup}
	\frac{\vert a(u , v)\vert}{ \,\, \Vert v\Vert_{V} \,\, \Vert u \Vert_{U} } \geq \beta > 0,
	%
	%
	\end{align}
	where $\Omega = I \times \Lambda_d$ and $\underset{u \in U}{\sup} \,\, \vert a(u , v)\vert>0$.
\end{thm}

\begin{thm}[well-posedness]
	\label{Thm: well-posedness_1D}
	For all $0<\alpha<1$, $\alpha \neq 1$, and  $1<\beta_j<2$, and $j=1,\cdots,d$, there exists a unique solution to \eqref{Eq: general weak form FPDE}, continuously dependent on  $f$, which belongs to the dual space of $U$.
\end{thm}
\begin{proof}
	Lemmas \ref{continuity_lem} (continuity) and \ref{inf_sup_d_lem} (stability) yield the well-posedness of weak form \eqref{Eq: general weak form FPDE} in (1+d)-dimension due to the generalized Babu\v{s}ka-Lax-Milgram theorem.
\end{proof}

Since the defined basis and test spaces are Hilbert spaces, and $U_N \subset U$ and $V_N \subset V$, we can prove that the developed Petrov-Gelerkin spectral method is stable and the following condition holds
\begin{align}
\label{Eq: inf sup-time}
\underset{u_N \neq 0 \in U_N}{\inf}\, \, \underset{v \neq 0 \in V_N}{\sup}
\frac{\vert a(u_N , v_N)\vert}{\Vert v_N\Vert_{V} \,\, \Vert u_N\Vert_{U}} \geq \beta > 0, 
\end{align}
with $\beta > 0$ and independent of $N$, where $\underset{u_N \in U_N}{\sup} \vert a(u_N , v_N)\vert>0$.


%
\section{\textbf{Numerical Results}}
\label{Sec: numerical results}
%
We investigate the performance of developed numerical methods by considering couple of numerical simulations. We compare MCS and PCM in random space discretization while using PG method in physical domain. We note that by several numerical examples, we make sure that the developed PG method is stable and accurate in solving each deterministic problem; the results are not provided here.

%
\subsection{Low-Dimensional Random Inputs}
%
As the first case, we consider a stochastic fractional initial value problem (IVP) with random fractional index by letting the diffusion coefficient to be zero, and also ignoring the additional random input and only taking $h(t)$ as the external forcing term. Therefore, we obtain
\begin{equation}
\label{Eq: SFIVP}
\prescript{}{0}{\mathcal{D}}_{t}^{\alpha(\xi)}  u(t;\xi) = h(t),
\end{equation}
subject to zero initial condition, where $u(t,\xi): (0,T] \times \Lambda \rightarrow \mathbb{R}$. We let $ u^{ext}(t) = \frac{\alpha}{2} \,\, t^{3+\frac{\alpha}{2}} $, $ h(t) = \prescript{}{0}{\mathcal{D}}_{t}^{\alpha(\xi)} u^{ext}(t) $ for each realization of $\alpha$. In this case, by choosing the tunning parameter $\tau$ in the temporal basis function to be $\frac{\alpha}{2}$, we can efficiently employ PG numerical scheme and also obtain the exact expectation by rendering the exact solution to be random with similar distribution as the random fractional index. Fig.\ref{Fig: MCM PCM SFIVP} shows the $L^2$-norm convergence rate of MCS and PCM in comparison of solution expectation with $\mathop{\mathbb{E}}^{ext}[u] = \mathop{\mathbb{E}}[u^{ext}]$. The results confirms converges rate of $0.5$ for MCS, while in PCM, the statistics of solution converges accurately very fast, using only few numbers of realizations. In this example, by ignoring the additional random input to the system, we take the advantage of having the exact random solution to be available. 
%
%
\begin{figure}[t]
	\centering
	\includegraphics[width=0.5\linewidth]{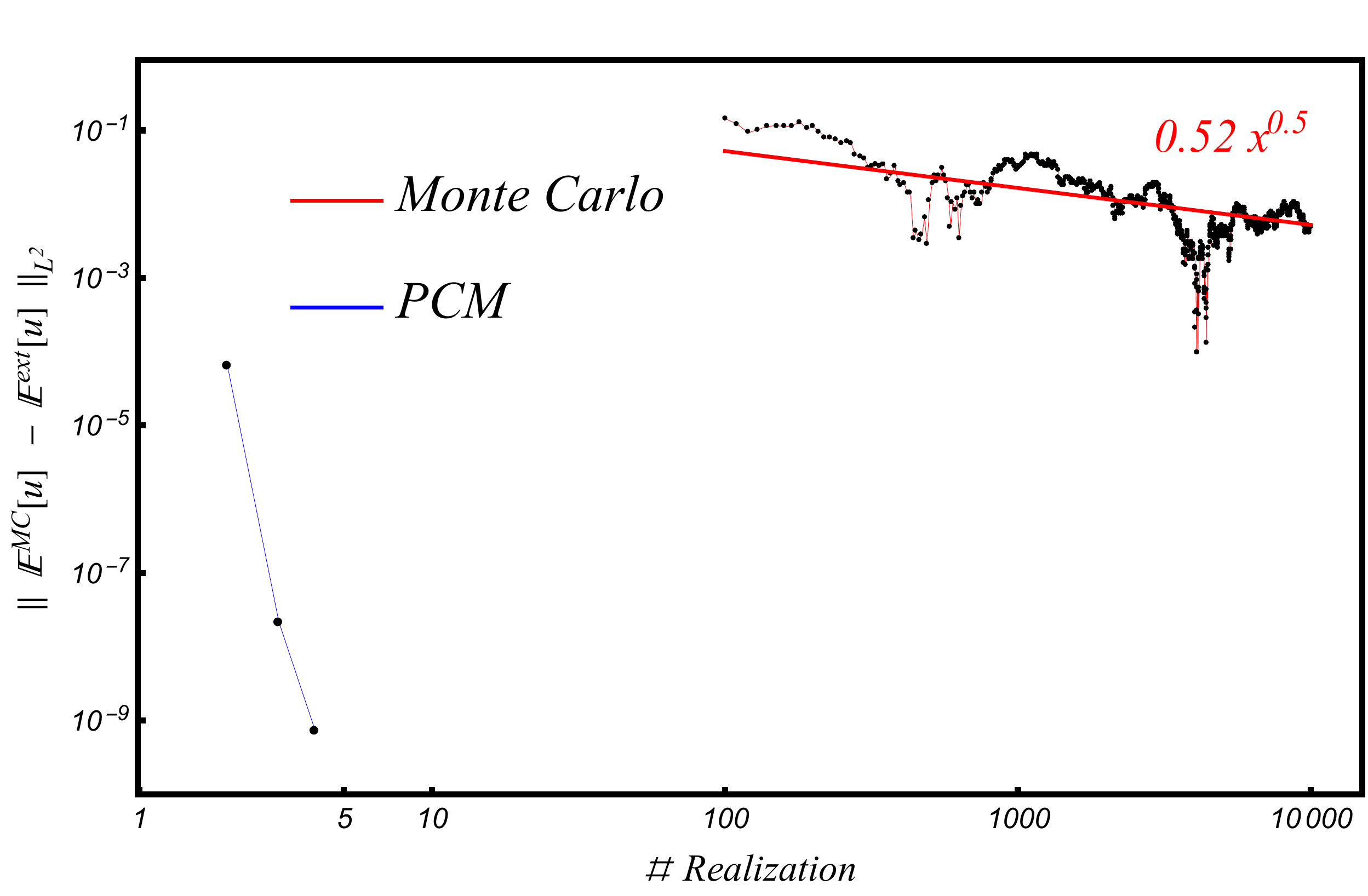}
	\caption{$L^2$-norm convergence rate of MCM and PCM for stochastic fractional IVP \eqref{Eq: SFIVP}.}
	\label{Fig: MCM PCM SFIVP}
\end{figure}
%
%

As another example, we also consider \eqref{Eq: SFIVP} with additional random input, expanded by KL expansion with $M=4$, as:
\begin{equation}
\label{Eq: SFIVP-2}
\prescript{}{0}{\mathcal{D}}_{t}^{\alpha(\xi)}  u(t;\boldsymbol{\xi}) = h(t) +  \sum_{k=1}^M \, a_k \, \sin\left(\frac{2 k \pi \, t}{T}\right) \, \xi_k,
\end{equation}
with two cases $h(t) = t^2$ and $h(t) = sin(\pi t)$. Fig.\ref{Fig: MCM PCM SFIVP-2} shows the mean value and variance of solution for $10^4$ sampling of MCS compared to $625$ realizations in PCM.
%
%
\begin{figure}[t]
	\centering
	\begin{subfigure}{0.45\textwidth}
		\centering
		\includegraphics[width=1\linewidth]{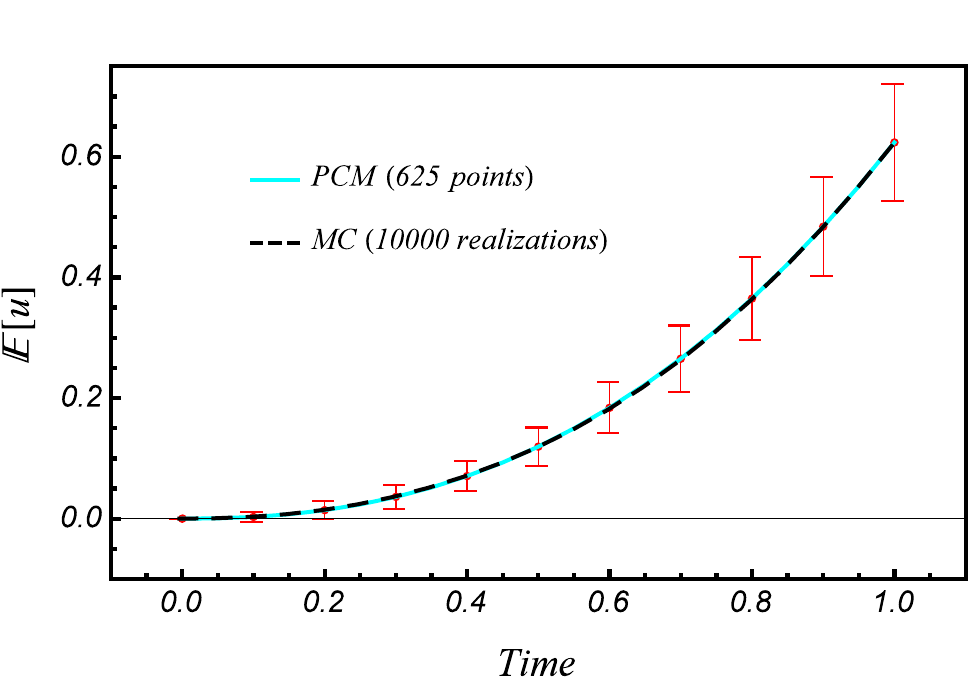}
	\end{subfigure}
	\begin{subfigure}{0.45\textwidth}
		\centering
		\includegraphics[width=1\linewidth]{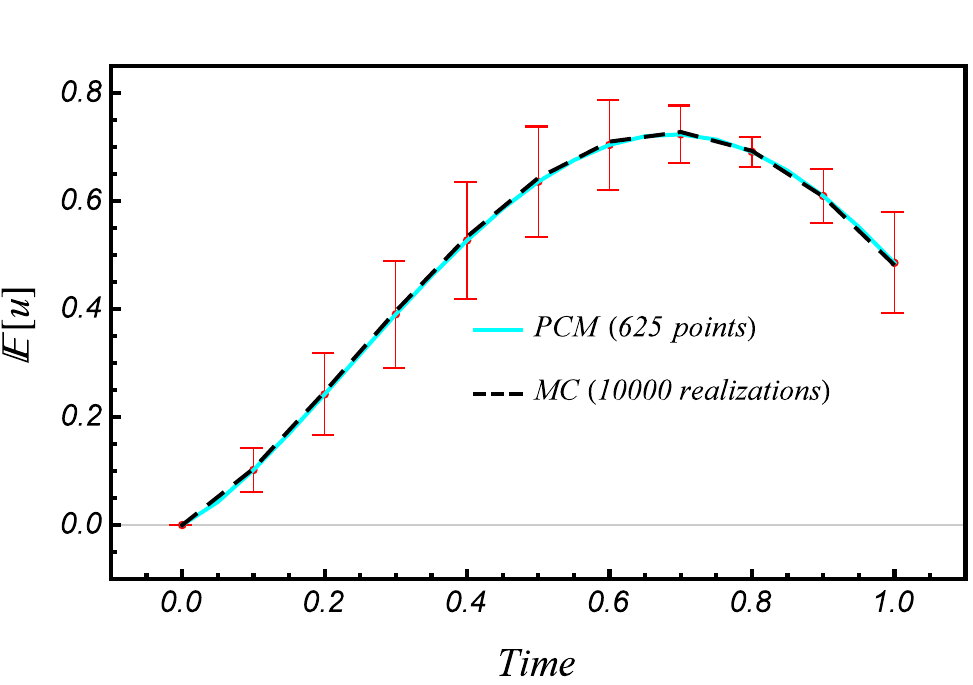}
	\end{subfigure}
	\caption{Expectation of solution to \eqref{Eq: SFIVP-2} with uncertainty (standard deviation) bounds, employing MCS and PCM for (left) $h(t) = t^2$ and (right) $h(t) = sin(\pi t)$. }
	\label{Fig: MCM PCM SFIVP-2}
\end{figure}
%
%


Moreover, we consider (1+1)-D one-sided SFPDE given in \eqref{Doob_momentum-2}, where $d=1$ and the diffusion coefficient is $k_l$. We ignore the additional random input and consider $h(t,x)$ as the only external forcing term. Therefore, we obtain
\begin{align}
\label{Eq: 1+1-d SFPDE no random noise}
&\prescript{}{0}{\mathcal{D}}_{t}^{\alpha(\xi_1)}  u(t,x;\boldsymbol{\xi}) 
+ k_l \prescript{}{-1}{\mathcal{D}}_{x}^{\, \beta(\xi_2)} u(t,x;\boldsymbol{\xi}) 
= h(t,x),
\end{align}
subject to zero initial/boundary conditions, where $u(t,x;\boldsymbol{\xi}): (0,T] \times (-1,1) \times \Lambda \rightarrow \mathbb{R}$, and the only random variables are the fractional indices $\alpha$ and $\beta$. We let $u^{ext}(t,x) = t^{3+\tau} \, \left((1+x)^{3+\mu}-\frac{1}{2}(1+x)^{4+\mu}\right)$, and choose $\tau = \alpha/2$ and $\mu = \beta/2 $. For each realization of $\alpha$ and $\beta$, we obtain the force function $h(t,x)$ by substituting the corresponding $u^{ext}$ to \eqref{Eq: 1+1-d SFPDE no random noise}. Defining $\mathop{\mathbb{E}}^{ext}[u] = \mathop{\mathbb{E}}[u^{ext}]$, Fig.\ref{Fig: MCM PCM SDPDE no noise} shows the $L^2$-norm convergence of solution expectation as compared to the exact expectation. We observe that PCM converges accurately with only few number of realizations.
%
%
\begin{figure}[t]
	\centering
	\includegraphics[width=0.5\linewidth]{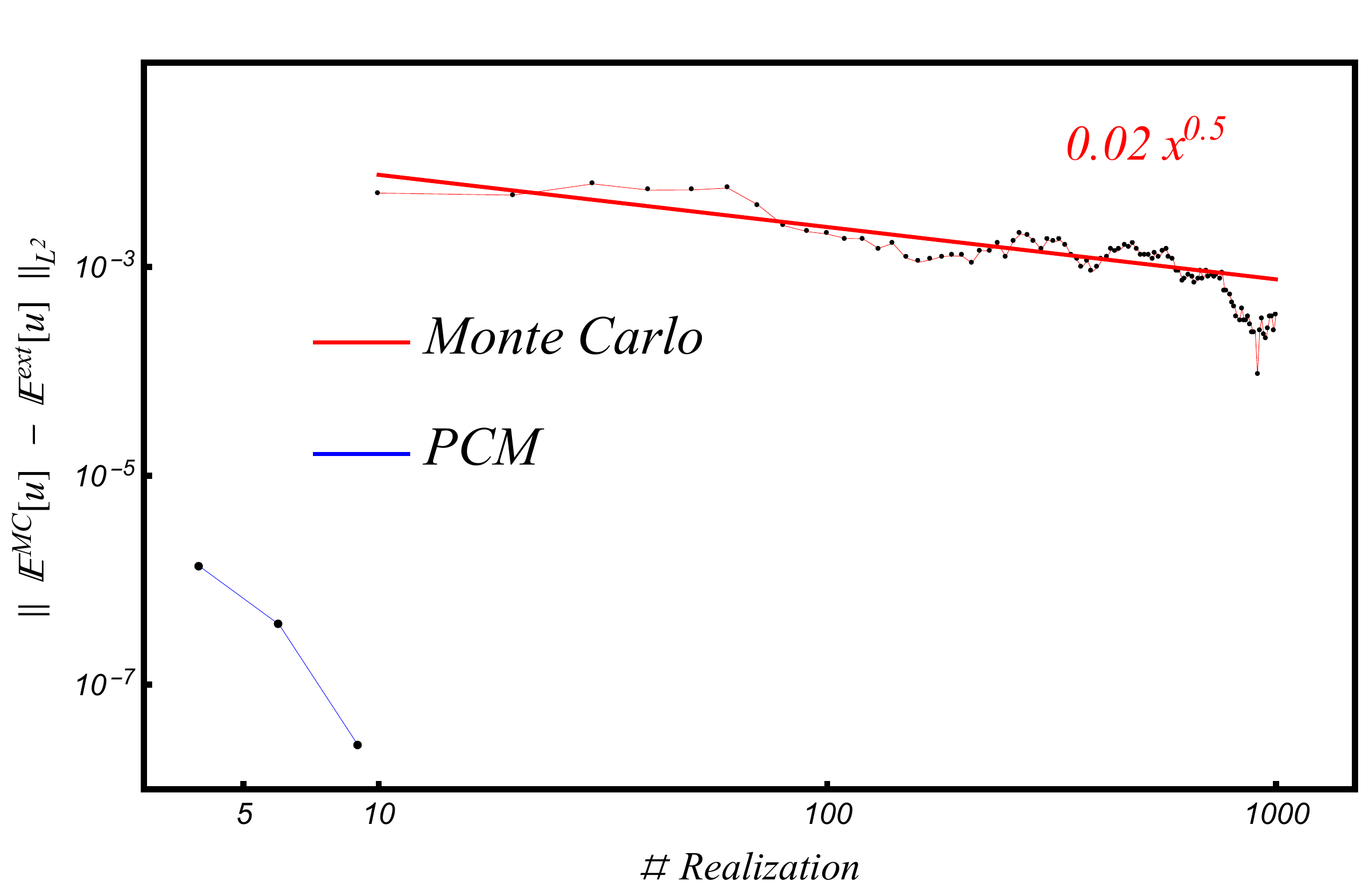}
	\caption{$L^2$-norm convergence rate of MCM and PCM for SFPDE \eqref{Eq: 1+1-d SFPDE no random noise}.}
	\label{Fig: MCM PCM SDPDE no noise}
\end{figure}
%
%

Considering additional random input, expanded by KL expansion with $M=4$, the problem can be recast as
\begin{align}
\label{1+1-d SFPDE with random noise}
&\prescript{}{0}{\mathcal{D}}_{t}^{\alpha(\boldsymbol{\xi})}  u(t,x;\boldsymbol{\xi}) 
+ k_l \prescript{}{-1}{\mathcal{D}}_{x}^{\, \beta(\boldsymbol{\xi})} u(t,x;\boldsymbol{\xi}) 
= h(t,x) +   \sum_{k=1}^M \, a_k \, \sin\left(\frac{2 k \pi \, t}{T}\right) \, \xi_k 
\end{align}
subject to zero initial/boundary conditions. Fig.\ref{Fig: MCM PCM SFPDE-2} shows the mean value and variance of solution for MCS and PCM at different times.

%
%
\begin{figure}[t]
	\centering
	\includegraphics[width=0.5\linewidth]{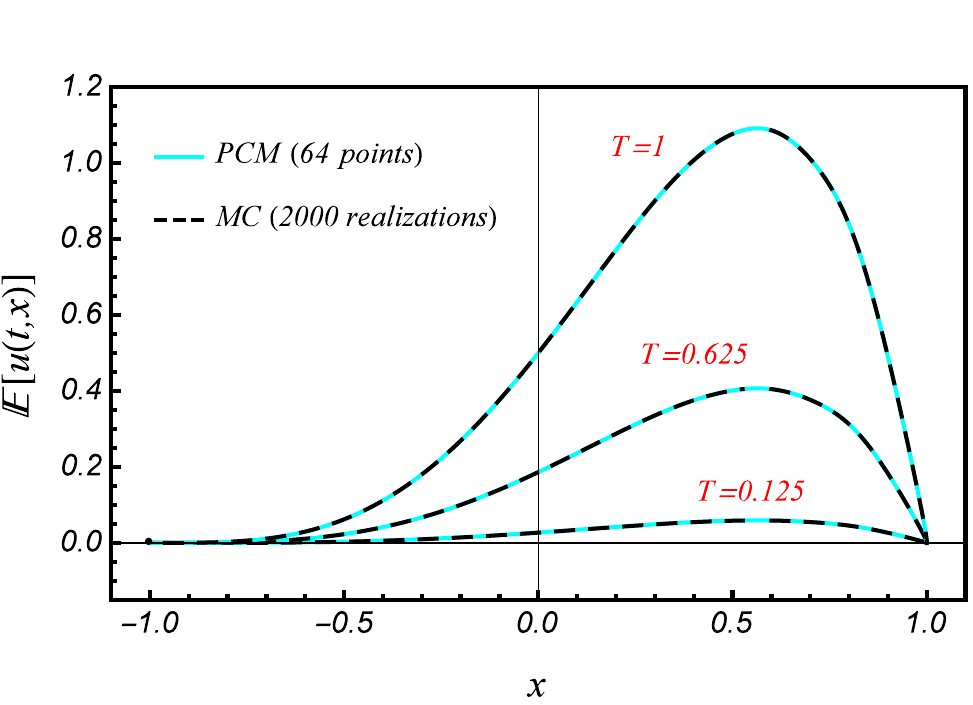}
	\caption{Expectation of solution to \eqref{1+1-d SFPDE with random noise}, employing MCS and PCM at $t=0.125, \, 0.625, \, 1$.}
	\label{Fig: MCM PCM SFPDE-2}
\end{figure}
%
%

\begin{rem}
We note that generally use of the sparse grid operators in obtaining solution statistics is more effective when dimension of the random space is higher than 6. Thus, in the numerical examples for low-dimensional random inputs, we employ the easy-to-implement tensor product nodal sets. 
\end{rem}

%
\subsection{Moderate- to High-Dimensional Random Inputs}
%
We render the problem with higher number of terms in KL expansion of random inputs in \eqref{1+1-d SFPDE with random noise} by choosing $M=10$ and $M=20$. This yields the dimension of random space $\mathcal{N} = 12$ and $\mathcal{N} = 22$, respectively. By employing the Smolyak sparse grid generator and using the developed PCM, we obtain the solution statistics. For each case, we generate the sparse grid on two levels $w=1$ and $w=2$, i.e. $A(1,12)$, $A(2,12)$, $A(1,22)$, and $A(2,22)$, where we let the higher resolution case be a benchmark value to the solution statistics, based on which we compute and normalize the error. We observe that for both cases $\mathcal{N} = 12$ and $\mathcal{N} = 22$, the normalized error in computing the expectation and standard deviation of solution are of orders $\mathcal{O}(10^{-7})$ and $\mathcal{O}(10^{-3})$, respectively.


%
\section{Summary and Discussion}
\label{Sec: Summary and Conclusion} 
%
We developed a mathematical framework to numerically quantify the solution uncertainty of a stochastic FPDE, associated with the randomness of model parameters. The stochastic FPDE is reformulated by rendering the problem with random fractional indices, subject to additional random noise. We used the truncated Karhunen-Lo\'eve expansion to parametrize the additive noise. Then, by employing a non-intrusive probabilistic collocation method (PCM), we propagated the associated randomness to the system response, by using Smolyak sparse grid generator to construct the set of sample point in the random space. We also formulated a forward solver to simulate the deterministic counterpart of the stochastic problem for each realization of random variables. We showed that the deterministic problem is mathematically well-posed in a weak sense. Furthermore, by employing Jacobi poly-fractonomials and Legendre polynomials as the temporal and spatial basis/test functions, respectively, we developed a Petrove-Galerkin spectral method to solve the deterministic problem in the physical domain. We also proved that the \text{inf-sup} condition holds for the proposed numerical scheme, and thus, it is stable. By considering several numerical examples with low- to high-dimensional random spaces, we examined the performance of our stochastic discretization. We showed that in each case, PCM converges very fast to a very high level of accuracy with very few number of sampling.


\bibliographystyle{siam}
\bibliography{SFPDE}

\end{document}